\documentclass{eptcs}
\usepackage[mathletters]{ucs}
\usepackage[utf8x]{inputenx}
\usepackage{amsmath,amsthm,amssymb,xspace,stmaryrd}
\usepackage{graphicx}
\usepackage[all,cmtip]{xy}

\usepackage{url}

\newcommand{\Idl}{\ensuremath{\mathrm{Idl}}}      
\newcommand{\RIdl}{\ensuremath{\mathrm{RIdl}}}    
\newcommand{\RSpec}{\ensuremath{\mathrm{RSpec}}}  
\newcommand{\Sh}{\ensuremath{\mathrm{Sh}}}        

\newcommand{\cC}[1]{\ensuremath{\mathcal{C}(#1)}} 
\newcommand{\Power}{\mathcal{P}}                  
\newcommand{\Cat}[1]{\ensuremath{\mathrm{\textbf{#1}}}} 
\newcommand{\Set}{\Cat{Set}\xspace}              
\newcommand{\cov}{\ensuremath{\vartriangleleft}}  
\newcommand{\uS}{\ensuremath{\underline{\Sigma}}} 
\newcommand{\Loc}{\Cat{Loc}\xspace}               
\newcommand{\spe}{\sqsubseteq}                    
\newcommand{\sier}{\mathbb{S}}                    
\newcommand{\ua}{\underline{A}}                   
\newcommand{\scal}{\ensuremath{\mathbb{Q}[i]}}    
\newcommand{\upperR}{\ensuremath{\overleftarrow{[0,\infty]}}}   
\newcommand{\aqftcov}{\mathcal{P}\ltimes\uS}        
\newcommand{\Str}[2]{\ensuremath{\mathbf{Str}_{#1}(#2)}} 
\newcommand{\Mod}[2]{\ensuremath{\mathbf{Mod}_{#1}(#2)}} 
\newcommand{\turndown}[1]{%
  \rotatebox[origin=c]{270}{\ensuremath#1}}       
\newcommand{\twoheaddownarrow}{\turndown{\twoheadrightarrow}}       
\newcommand{\dd}{\twoheaddownarrow}               
\newcommand{\turnround}[1]{%
  \rotatebox[origin=c]{180}{\ensuremath#1}}       
\newcommand{\wi}{\eqslantless}                    
\newcommand{\downset}{\ensuremath{\mathop{\downarrow}}}  
\newcommand{\upset}{\ensuremath{\mathop{\uparrow}}}  
\newcommand{\oc}{\ensuremath{[\mathrm{set}]}}     

\newcommand{\frk}{\mathfrak}
\renewcommand{\coprod}{\turnround{\prod}}
\newcommand{\ev}{\mathrm{ev}}                     
\newcommand{\id}{\mathrm{id}}                     
\newcommand{\opens}{\mathcal{O}}                  
\newcommand{\pt}{\mathrm{pt}}                     

\newtheorem{theorem}{Theorem}
\newtheorem{lemma}[theorem]{Lemma}
\newtheorem{proposition}[theorem]{Proposition}

\newtheorem{definition}[theorem]{Definition}

\begin{document}
 \title{Gelfand spectra in Grothendieck toposes\\
  using geometric mathematics}
 \def\titlerunning{Gelfand spectra in Grothendieck toposes}
 \def\authorrunning{Bas Spitters, Steven Vickers \& Sander Wolters}
 \author{Bas Spitters
   \institute{VALS-LRI, Universit\'e Paris-Sud/INRIA Saclay}
   \email{bas.spitters@inria.fr}
 \and Steven Vickers
   \thanks{Supported by the UK Engineering and Physical Sciences Research Council,
     on the project EP/G046298/1
     ``Applications of geometric logic to topos approaches to quantum theory''.}
   \institute{School of Computer Science, University of Birmingham,\\
     Birmingham, B15 2TT, UK.}
   \email{s.j.vickers@cs.bham.ac.uk}
 \and Sander Wolters
 \thanks{Supported by N.W.O. through project 613.000.811.}
\institute{Radboud University Nijmegen, IMAPP}
   \email{s.wolters@math.ru.nl}
 }
\maketitle
\begin{abstract}
In the (covariant) topos approach to quantum theory by Heunen, Landsman and Spitters,
one associates to each unital C*-algebra $A$ a topos $\mathcal{T}(A)$
of sheaves on a locale and a commutative C*-algebra $\ua$ within that topos.
The Gelfand spectrum of $\ua$ is a locale $\uS$ in this topos,
which is equivalent to a bundle over the base locale.
We further develop this external presentation of the locale $\uS$,
by noting that the construction of the Gelfand spectrum in a general topos
can be described using geometric logic.
As a consequence, the spectrum, seen as a bundle, is computed fibrewise.

As a by-product of the geometricity of Gelfand spectra, we find an explicit external description of the spectrum whenever the topos is a functor category. As an intermediate result we show that locally perfect maps compose, so that the externalization of a locally compact locale in a topos of sheaves over a locally compact locale is locally compact, too.
\end{abstract}

\section{Introduction} \label{sec:intro}

The main subject of this paper is the interplay between geometric logic and topos-theoretic approaches to C*-algebras (motivated by quantum theory).
In particular, we consider the approach of Heunen, Landsman and Spitters \cite{CaspersHLS:IntQLnLS,qtopos,Bohrification_ql,Bohrification},
although some of the ideas and techniques in this paper may turn out to be of interest to the related approach by
Butterfield, Isham and D\"oring~\cite{butterfieldisham1, butterfieldisham2, butterfieldisham4, butterfieldisham3,
DoeringIsham:WhatThingTTFP} as well; see~\cite{Wolters} for a comparison. We are mainly
interested in the spectral object of the topos approach and its construction using geometric logic.

We assume that the reader is familiar with the basics of topos and locale theory.%
\footnote{
  The standard reference~\cite{maclanemoerdijk92} contains much of the needed material.
  In particular, Chapter II of this book demonstrates how sheaves on a topological space can be seen as bundles.
  Chapter III gives an introduction to Grothendieck toposes.
  Section VII.1 contains background information on geometric morphisms.
  Finally Chapter IX gives all background information on locales.
  At certain points in this paper, in particular in Subsection~\ref{sec:functor},
  sheaf semantics is used.
  The relevant background material can be found in Chapter VI of~\cite{maclanemoerdijk92}.
  Another standard reference is the massive work~\cite{Elephant1,Elephant2}.
}
All toposes are assumed to be Grothendieck toposes,
and, in particular, every topos has a natural numbers object (NNO).
In its general form, the theory of Grothendieck toposes is that of bounded toposes over some base topos
$\mathcal{S}$ that embodies the ambient logic.
We shall rarely need to be explicit about $\mathcal{S}$,
but our techniques will be valid for $\mathcal{S}$ an arbitrary elementary topos with NNO and thus have
wide constructive applicability.

The paper is split into two sections.
Section~\ref{sec:geo} gives background information on geometric logic and, more importantly,
the practical impact of the geometric mathematics that develops from it.
For this we can give a first definition as topos-valid constructions on sets
(understood as objects of a topos) which are preserved
by inverse image functors of geometric morphisms.
This is the geometric mathematics of \emph{sets}, as objects of a topos.

More profoundly, we can also consider geometricity of topos-valid constructions on locales,
and much of our ability to do this follows from two results in \cite{joyaltierney84}.
The first provides a localic version of a well known result from point-set topology,
that sheaves over $X$ are equivalent to local homeomorphisms with codomain $X$.
On the one hand, the sheaves are the ``sets'' in the topos $\Sh(X)$ of sheaves,
or the discrete locales -- the frames are the powerobjects.
On the other hand, the local homeomorphisms can be understood as the fibrewise
discrete bundles over $X$, where we understand ``bundle'' in the very general sense of
locale map with codomain $X$.
Applying an inverse image functor $f^{\ast}$ to the sheaf corresponds to
pulling back the bundle along $f$.
Hence geometricity of a construction on the sheaves corresponds to preservation
under pullback of the corresponding construction on the bundles.

The second result in \cite{joyaltierney84} is that internal frames in $\Sh(X)$ are dual
to localic bundles over $X$.
This immediately allows us to extend our definition of geometricity to constructions
on locales, namely as preservation under pullback of the bundle constructions.
Since the fibres are got as pullbacks along points, the geometric mathematics works
fibrewise and provides a fibrewise topology of bundles.
This idea has already been explored in an ad hoc fashion in point-set topology --
see, e.g., \cite{James:FibrewiseTop} --, and the notion of geometricity makes it
much more systematic when one combines point-free topology with toposes.

It should be noted that pullback of a bundle along a map $f$ is \emph{not} achieved
by applying $f^{\ast}$ to the internal frame.
This is already clear in the fibrewise discrete (local homeomorphism) case,
since $f^{\ast}$ does not preserve powerobjects, nor frame structure in general.
\cite{joyaltierney84} define a different functor $f^{\#}$ that transforms frames to frames,
but in practice it is often convenient to bypass the frames altogether and instead
use presentations of them by generators and relations.
\cite{PPExp} shows that applying $f^{\ast}$ to the presentation corresponds to pullback
of the bundle, even though the middle step of presenting a frame is not geometric.

The practical effect of switching to generators and relations is that a locale is described
by means of a geometric theory whose models are the (generalized) points of the locale.
Hence we explicitly describe the points rather than the opens,
though the nature of geometric theories ensures that the topology is described implicitly at the
same time, by presenting a frame.
Geometricity has the effect of restoring the points to point-free topology,
and allows us to define maps $f:Y\to X$ in two very intuitive ways as
geometric transformations:
as a map, $f$ transforms points of $Y$ to points of $X$,
while as a bundle it transforms points of $X$ to locales, the fibres,
and defines the bundle locale $Y$ at the same time.
This is explained in more detail in \cite{Vickers:ContIsGeom}.

This switch of emphasis can be disconcerting if one thinks that a locale \emph{is} its frame.
However, calculating an internal frame is error prone
-- experience shows this even in the simpler case of presheaf toposes --,
and one of the central messages of this paper is that the geometric methods often allow us to avoid
the frame.
On the other hand, sometimes the frame is still needed.
Later on we show how to exploit local compactness to give a geometric calculation of it.
Specifically, we prove a result about exponentiability of objects in a topos,
which entails that if $X$ is a locally compact locale and $\underline{Y}$
is a locally compact locale in $\Sh(X)$, with bundle locale $Y$ over $X$, then $Y$ is locally compact.

Section~\ref{sec:quantum} applies the previous discussion about geometricity to the HLS topos approach of quantum theory.
Since the topos is that of sheaves over a locale, the spectrum can be understood as a bundle
\cite{FRV:Born},
in which each fibre is the spectrum of a commutative C*-algebra,
and the geometric approach will emphasize this.
In particular, we concentrate on the external description of the bundle locale,
which is the phase space of the topos approach.
This phase space is the Gelfand spectrum of a unital commutative C*-algebra, internal to a functor category.
The first two subsections seek a fully geometric description of the construction of this spectrum.
The basic problem is that, geometrically, a C*-algebra has a natural topology and therefore should be
a locale rather than a set.
As far as we know, there is not yet a satisfactory account of Gelfand duality that deals with
the C*-algebra in this way.
\cite{banaschewskimulvey06} defines C*-algebra in a way that is topos-valid but not geometric -- it uses a non-geometric characterization of
completeness.
We show how to replace the complete C*-algebra by a geometric stucture that yields the same spectrum,
and following \cite{coquand05} we quickly reduce to the spectrum of a normal distributive lattice,
a general lattice theoretic way to present compact regular locales.
In the subsection that follows these ideas are used to find an explicit description of the phase space
in question. This is subsequently generalized to a description of Gelfand spectra
of arbitrary unital commutative C*-algebras in functor categories. In the final subsection we briefly
consider generalizing the topos approach to C*-algebras to algebraic quantum field theory.

\section{Geometricity}\label{sec:geo}

This section is divided into four parts.
In Subsection~\ref{sub:geometriclogic} we briefly discuss
geometric logic and associated mathematical constructions.
An important point is that the constructions that can be expressed by geometric mathematics
coincide with the constructions that are preserved when pulled back along the inverse image
functor of a geometric morphism.
Subsection~\ref{sub:geometricmath} shifts attention from geometric logic to geometric mathematics (in particular by using geometric type constructs). We discuss geometricity of topos-valid constructions on sets (in the generalized sense of objects in a topos) and on locales. Subsequently, in Subsection~\ref{sub:toppreshsh} the geometricity of such constructions is further analysed when these constructions are carried out in toposes that are functor categories and in toposes of sheaves on a topological space. Finally, in Subsection~\ref{sub:expo} we discuss exponentiability of locales.

The results of this section are important in two ways.
First, they allow us to talk about the opens, the elements of a frame
(which is itself not a geometric concept), in a geometric way.
Second, it entails that the external description of the Gelfand spectrum investigated in the next section is a locally compact locale.

\subsection{Geometric logic}\label{sub:geometriclogic}

We briefly describe geometric theories and their interpretations in toposes. The discussion is brief because the work in Section~\ref{sec:quantum} concentrates on geometric mathematics, as described in Subsection~\ref{sub:geometricmath}, rather than geometric logic as presented in this subsection. The reader is advised to browse rather than carefully read this subsection, and to consult it at a later stage, if needed.

In this subsection we follow
the discussion in \cite[Section D1]{Elephant1} where a much more detailed presentation is given.
Section 3 of \cite{LocTopSp} also treats this material, with more emphasis on geometric logic.
Another good source is \cite[Chapter X]{maclanemoerdijk92}, but
note that what is called \emph{geometric} there is called \emph{coherent} in \cite{Elephant1,Elephant2} and \cite{LocTopSp}:
the difference is that ``coherent'' forbids infinite disjunctions.
Finally, \cite{Vickers:ContIsGeom} provides an alternative introduction to some of the techniques presented in this section.

The language of geometric logic is that of an infinitary, first-order,
many-sorted predicate logic with equality%
\footnote{
  We shall not explicitly consider the geometric types of e.g.~\cite[Subsection 3.4]{LocTopSp}.
}.
We start with a first-order \textbf{signature} $\Sigma$
(where the notation $\Sigma$ has nothing to do with the spectrum $\Sigma$ in the following
sections).
The signature consists of a set $S$ of \textbf{sorts}, a set $F$ of \textbf{function symbols} and a set $R$
of \textbf{predicate symbols}. Each function symbol $f\in F$ has a type, which is a non-empty finite list of sorts
$A_{1},...,A_{n},B$. We write this as $f:A_{1}\times...\times A_{n}\to B$, in anticipation of the categorical
interpretation of the function symbols. The number $n$ is called the \textbf{arity} of the function symbol.
If the arity
of a function symbol $f$ is $0$, i.e., when its type is a single sort $B$ (we will write $f:1\to B$), $f$ is called a
\textbf{constant}. Each predicate $P\in R$ also has a type $A_{1},...,A_{n}$,
where we also allow the empty list ($n=0$)
as a type. We will write $P\subseteq A_{1}\times...\times A_{n}$, again in anticipation of the categorical
representation. A predicate with arity equal to $0$, written as $P\subseteq 1$, will be called a \textbf{propositional
symbol}.

Given the signature $\Sigma$, a \textbf{context}%
\footnote{This should not be confused with the non-logical notion of context used in the topos
approaches to quantum theory.}
is a finite set of variables, each associated with some sort.
It is customary to write the set as a vector $\vec{x}=(x_{1},...,x_{n})$ of distinct variables.
Using $\Sigma$ and the variables we can define terms and formulae over $\Sigma$ in that context.
\textbf{Terms}, which all have a sort assigned to them, are inductively defined as follows:
\begin{itemize}
\item Each variable $x$ of sort $A$ is a term of sort $A$.
\item If $f:A_{1}\times...\times A_{n}\to B$ is a function symbol, and $t_{1},..., t_{n}$ are terms of sort
$A_{1},...,A_{n}$, then $f(t_{1},...,t_{n})$ is a term of sort $B$. In particular, every constant is a term.
\end{itemize}

We can now construct geometric formulae in that context inductively as follows.
They form the smallest class which is closed under the following clauses:
\begin{enumerate}
\item
  If $R\subseteq A_{1}\times...\times A_{n}$ is a predicate, and $t_{1},...,t_{n}$ are terms of
  sort $A_{1},...,A_{n}$, then $R(t_{1},...t_{n})$ is a formula.
  For the particular case $n=0$, even in the empty context every propositional symbol is a formula.
\item
  If $s$ and $t$ are terms of the same sort, then $s=t$ is a formula.
\item
  Truth $\top$ is a formula. If $\phi$ and $\psi$ are formulae,
  then so is the conjunction $\phi\wedge\psi$.
\item
  Let $I$ be any (index) set and for every $i\in I$, let $\phi_{i}$ be a formula.
  Then $\bigvee_{i\in I}\phi_{i}$ is a formula.
\item
  If the variable $u$ is not in the context $\vec{x}$, and $\phi$ is a formula in context
  $\vec{x}\cup\{u\}$, then $(\exists u)\phi$ is a formula in context $\vec{x}$.
\end{enumerate}
Note that implication, negation and the universal quantifier are not allowed
as connectives in the construction of geometric formulae.

To show the context explicitly
we shall denote a formula or term in context by $\vec{x}.\phi$ or $\vec{x}.t$.
Note that a formula or term does not have to use all the variables in its context -- some may be unused.

We will also consider sequents $\phi\vdash_{\vec{x}}\psi$, where $\phi$ and $\psi$ are geometric formulae in context $\vec{x}$. We will think of the sequent as expressing that $\psi$ is a logical
consequence of $\phi$ in context $\vec{x}$. A \textbf{geometric theory} $\mathbb{T}$ over $\Sigma$ is simply a set of
such sequents $\phi\vdash_{\vec{x}}\psi$.

The next step is to consider interpretations of geometric theories in toposes.
The discussion will be very brief,
but the details can be found in the references stated at the beginning of this subsection.
Let $\mathcal{E}$ be a topos
(actually any category that has finite products would suffice for defining the structures,
and also assuming pullbacks makes the definition of a homomorphism of $\Sigma$-structures nicer).
Given a signature $\Sigma$, a \textbf{$\Sigma$-structure}
$M$ in $\mathcal{E}$ is defined as follows. For every sort $A$ in $\Sigma$ there is an associated object $MA$ in
$\mathcal{E}$. For every function symbol $f:A_{1}\times... A_{n}\to B$ there is an arrow
$Mf \colon MA_{1}\times...\times MA_{n}\to MB$ in $\mathcal{E}$.
A constant $c:1\to B$ is interpreted as an arrow $Mc:1\to MB$, where $1$ denotes the
terminal object of $\mathcal{E}$. A predicate $R$ of type $A_{1}\times... \times A_{n}$ is interpreted as a monic arrow
$MR\rightarrowtail MA_{1}\times...\times MA_{n}$.

If $M$ and $N$ are $\Sigma$-structures in $\mathcal{E}$, then a \textbf{homomorphism of $\Sigma$-structures}
$h$ is defined as follows. For each sort $A$ in $\Sigma$ there is an arrow $h_{A}: MA\to NA$. For each function symbol
$f:A_{1}\times...\times A_{n}\to B$, we demand $h_{B}\circ Mf=Nf\circ(h_{A_{1}}\times...\times h_{A_{n}})$. If
$R\subseteq A_{1}\times...\times A_{n}$ is a predicate, then we demand that $MR\subseteq(h_{A_{1}}\times...\times
h_{A_{n}})^*(NR)$ holds as subobjects of $MA_{1}\times...\times MA_{n}$, where the right hand side means pulling the
monic arrow $NR\rightarrowtail NA_{1}\times...\times NA_{n}$ back along $h_{A_{1}}\times...\times h_{A_{n}}$.

The $\Sigma$-structures in a topos $\mathcal{E}$ and their homomorphisms define a category
$\Str{\Sigma}{\mathcal{E}}$.
Let $F:\mathcal{E}\to\mathcal{F}$ be a functor between toposes. $F$ need not come
from a geometric morphism, but we do assume it to be left exact.
Any such functor
induces a functor $\Str{\Sigma}{F}: \Str{\Sigma}{\mathcal{E}}\to\Str{\Sigma}{\mathcal{F}}$ in a straightforward way.

The next step is to introduce models of a geometric theory $\mathbb{T}$ over $\Sigma$.
In order to do this we
need to interpret terms and formulae-in-context for a $\Sigma$-structure in $\mathcal{E}$. This can be done inductively,
in much the same way as in using the internal language of the topos. Details can be found in
\cite[Subsection D1.2]{Elephant1}. For a given $\Sigma$-structure $M$, the end result is that a formula-in-context $\vec{x}.\phi$, where
$\vec{x}=(x_{1},...,x_{n})$ are variables with associated sorts $A_{1},...,A_{n}$, is interpreted as a subobject
$M(\vec{x}.\phi)$ of $MA_{1}\times...\times MA_{n}$. A $\Sigma$-structure $M$ in a topos $\mathcal{E}$ is called a
\textbf{model} for a geometric theory $\mathbb{T}$ if for every sequent $\phi\vdash_{\vec{x}}\psi$ in $\mathbb{T}$ we
have $M(\vec{x}.\phi)\subseteq M(\vec{x}.\psi)$, where we view the interpretation of the formulae as subobjects of
$MA_{1}\times...\times MA_{n}$.
We write $\Mod{\mathbb{T}}{\mathcal{E}}$ for the full subcategory of $\Str{\Sigma}{\mathcal{E}}$
whose objects are the models of $\mathbb{T}$.

Although any left-exact functor $F:\mathcal{E}\to\mathcal{F}$ induces a functor
$\Str{\Sigma}{F}$ between the associated categories of $\Sigma$-structures,
in general $F$ does not preserve those ``geometric constructions'' used to interpret formulae,
and consequently $\Str{\Sigma}{F}$ does not restrict to model categories.
On the other hand, if $F=f^{\ast}$, the inverse image of a geometric morphism $f$,
then $\Str{\Sigma}{F}$ does restrict to a functor
$\Mod{\mathbb{T}}{F}: \Mod{\mathbb{T}}{\mathcal{E}}\to\Mod{\mathbb{T}}{\mathcal{F}}$.
A nice proof of the restriction is given in \cite[Section X.3]{maclanemoerdijk92}.
Although only finite joins are considered there, this proof is in particular
interesting because it also treats (non-geometric) formulae that use implication
and the universal quantifier, and shows that models of a theory using this additional structure
are only preserved by the inverse image functor of a
geometric morphism when the geometric morphism is open.
In general, we shall describe a topos-valid construction as \textbf{geometric}
if it is preserved (up to isomorphism) by inverse image functors.

One important kind of mathematical structure that can be expressed by
geometric logic is finitary algebraic theories, such as monoids, (Abelian) groups, rings and lattices.
Their axioms all take the shape $\top\vdash_{\vec{x}}s=t$ where $s$ and $t$ are terms.
These theories are clearly geometric, as are
many-sorted algebraic theories such as pairs of rings and modules.
The theory of partially ordered sets, but also ordered groups or ordered rings, is geometric.
More general examples are local rings, finite sets and (small) categories. See
\cite[Example D1.1.7]{Elephant1} for more details.

Sometimes it can be hard to see whether some given structure is geometric.
Just because it is presented by a theory that uses $\forall$ or $\Rightarrow$ in formulae,
does not mean that there is no equivalent geometric theory that describes the structure.

One example of structure that is definitely not geometric is that of frames:
complete lattices with binary meet distributing over arbitrary joins, and homomorphisms
preserving joins and finite meets.
This will be important in Section~\ref{sec:quantum},
where we will look at spectra of commutative unital C*-algebras in toposes:
the spectrum is described by a frame.
The non-geometricity of frames is shown by the fact that the inverse image functor
of a geometric morphism does not necessarily map frames to frames.
C*-algebras provide another example: inverse image functors need not map C*-algebras to C*-algebras.
Still, we can say a lot about the spectrum of a commutative C*-algebra by paying attention to
geometric constructions, as we will see in Subsection~\ref{sub:bohr}.

Although there is not a single geometric theory whose models are the frames,
they are closely related to the important class of \emph{propositional} geometric theories --
those with no declared sorts so that the signature consists entirely of propositional symbols.
Each such theory can be straightforwardly translated into a presentation of a frame
by generators and relations, with propositional symbols and axioms
becoming the generators and relations,
and then the points of the corresponding locale are equivalent to the models of the theory.

Defining the notion of geometric theory and their models in toposes is just the start of
an introduction to geometric logic.
The next step would be to discuss the rules of inference. However, for the purposes of this paper we have treated enough and the reader interested
in the rules of inference is directed to \cite[Section D1.3]{Elephant1}.

More important for us will be the way that geometric theories are used to describe the
\textbf{generalized points} of a Grothendieck topos $\mathcal{F}$, the geometric morphisms
$\mathcal{E}\to \mathcal{F}$ to $\mathcal{F}$ from another topos.
Whenever we refer to \textbf{points} we mean \textbf{generalized} points in the above sense, and not the smaller class of \textbf{global points}, i.e., points where $\mathcal{E}=\Set$.
A key fact is that for each $\mathcal{F}$ there is a geometric theory $\mathbb{T}$ such that
the points of $\mathcal{F}$ are equivalent to the models of $\mathbb{T}$ in $\mathcal{E}$.
$\mathcal{F}$ is said to \textbf{classify} the theory $\mathbb{T}$.
Conversely, every geometric theory has a classifying topos.

The same techniques can also be applied to \textbf{localic} toposes, i.e., those of the form
$\Sh X$, where $X$ is a locale. They classify propositional geometric theories.

The geometric approach will continually ask -- What are the points? What theory does it classify?
This will come as a shock to those used to thinking of the points as insufficient -- as,
indeed, the global points are -- and calculating concretely with the topos or the frame.

\subsection{Geometric mathematics} \label{sub:geometricmath}

In Subsection~\ref{sub:geometriclogic} we gave formal definitions of geometric theories and their
interpretation in toposes.
In what follows, however, we shall make little use of formal geometric theories.
This is not just for the usual reason, that in practical mathematics it is tedious to work formally.
We shall also be extensively using type-theoretic features of the logic that have not been fully formalized.
Our informal treatment will be what we describe as ``geometric mathematics''.

It has two levels. A geometric mathematics of ``sets'' (understood as objects of a topos) consists of those
structures, constructions and theorems that are preserved by inverse image parts of geometric morphisms.
The next level is a geometric mathematics of locales, and we shall explain geometricity here
in terms of bundles.

A key property of geometric logic, distinguishing it from finitary first-order logics,
is that the infinitary disjunctions allow us to characterize some constructions (necessarily geometric)
uniquely up to isomorphism by using geometric structure and axioms.
These constructions include $\mathbb{N}, \mathbb{Z}$ and $\mathbb{Q}$ (though not $\mathbb{R}$ and
$\mathbb{C})$, free algebras more generally, colimits and finite limits.
Using these constructions informally, but in the knowledge that they could be formalized within
geometric theories, allows us to deal with geometric theories and their models in a way
that is closer to a ``geometric mathematics'' than to logic.
We shall do this extensively in the rest of the paper.

To explain how this extends to constructions on locales,
we shall need to recall the equivalence between internal locales and bundles.
For simplicity we assume we are working in a topos $\mathcal{E}=\Sh(X)$,
where $X$ is a locale in $\Set$,%
\footnote{
  -- or an arbitrary base topos $\mathcal{S}$, elementary with NNO.
  We do not assume classical sets.
}
though the idea works more generally.
Then the category $\mathbf{Loc}_{\Sh(X)}$ of locales in $\Sh(X)$
(defined as the dual of the category of internal frames) is
equivalent to the slice category $\mathbf{Loc}/X$,
of locales in $\Set$ over $X$
(\cite{joyaltierney84} or \cite[Section C1.6]{Elephant1}).
A locale $\underline{Y}$ in the topos $\mathcal{E}=\Sh(X)$ can be represented by a locale map $p:Y\to X$.
We will call a map a \textbf{bundle} when we view it in this way,
and think of it as a family of locales, the fibres, parametrized by (generalized) points of $X$.
Let $x:X'\to X$ be a point of $X$, and consider the pullback diagram
\[
  \xymatrix{ X'\times_{X}Y \ar[d] \ar[r] & Y \ar[d]^{p}\\
  X' \ar[r]_{x} & X}
\]
If $X' = 1$, so $x$ is an ordinary (global) point of $X$, then the pullback $X'\times_{X}Y$ is the fibre of
$p$ over $x$.
When we generalize to arbitrary $X'$ then we may think of the pullback as the generalized fibre of $p$
over the generalized point $x$.

If we define toposes $\mathcal{E}=\Sh(X)$ and $\mathcal{F}=\Sh(X')$, then $x$ becomes a geometric
morphism $x: \mathcal{F}\to\mathcal{E}$.
As we have noted earlier, the inverse image part $x^{*}: \mathcal{E}\to\mathcal{F}$
does not preserve frame structure.
Nonetheless, there is a functor
$x^{\#}:\mathbf{Frm}_{\mathcal{E}}\to\mathbf{Frm}_{\mathcal{F}}$
whose action on the corresponding bundles is by pullback along $x$.
Moreover it is a right adjoint to
$x_{\ast}: \mathbf{Frm}_{\mathcal{F}}\to\mathbf{Frm}_{\mathcal{E}}$,
the direct image part of $x$, which,
unlike $x^{\ast}$, does preserve frames.

For each object $U$ in a topos $\mathcal{E}$ there is a corresponding \textbf{discrete locale}
whose frame is the powerobject $\Power U$.
The corresponding bundles are the \'etale bundles, or local homeomorphisms (\cite[Chapter II]{maclanemoerdijk92}),
whose fibres are normally called stalks.
If $F$ is a geometric morphism with codomain $\mathcal{E}$ then
$f^{\#}(\Power U)\cong \Power f^{*}(U)$,
and it follows that applying $f^{*}$ corresponds to pulling back the bundle of the discrete locale.
Our notion of geometricity with regard to constructions on objects of toposes consisted of  preservation under
inverse image functors. If we translate the construction into one on the bundles, we see that
geometricity comes to mean preservation under pullback.

In this form we have a notion of geometricity that can also be understood with regard to locales,
once they are interpreted as bundles.
Another way to say the same thing is that the mathematics works fibrewise (for generalized fibres),
since fibres are just pullbacks of bundles.

We often make use of the following observation. Let $p:Y \to X$ be an arbitrary localic bundle: we take it that there is an internal locale
in $\Sh X$ corresponding to it.
If $y: W \to Y$ is a (generalized) point of $Y$, then we obtain a point $x = py$ of X
and a map $y': W \to x^{*}Y$ to the pullback $x^{*}Y$:
\[
  \xymatrix{
    W \ar@{=}[ddr] \ar@/^/[drr]^y \ar[dr]^{y'}  \\
      & x^{\ast}Y \ar[d] \ar[r]    & Y \ar[d]^p               \\
      & W   \ar[r]^x               & X
  }
\]
Thus we may describe the points of $Y$ as the pairs $(x,y')$, where $x$ is a point of $X$
and $y'$ is a point of the fibre $x^{*}Y$.

This description is geometric, because the structure and properties of the diagram are preserved by any
change of base (pullback along maps into $W$).
It follows that, for a geometric theory for $Y$, we need one whose models are the points $(x,y')$.

We shall also need to consider the specialization order $\spe$ on $Y$,
and for this the analysis needs to be more refined.
The straightforward part is that within a single fibre $x^{*}Y$
we have $y'_1 \spe y'_2$ iff $(x,y'_1)\spe(x,y'_2)$.
The $\Rightarrow$ direction follows immediately from the map $x^{*}Y\to Y$,
while the converse follows from the fact that in $\Loc$ the pullback property determines
the hom-\emph{po}sets into a fibre, as well as the homsets.
Also, if $(x_1,y'_1)\spe(x_2,y'_2)$ then $x_1\spe x_2$ by the map $p$.

For a more precise analysis of the specialization between fibres we shall assume
that $p$ is an \emph{opfibration} in the 2-categorical sense.
(A similar analysis applies when $p$ is a fibration.)%
\footnote{
  There are also topological notions of fibration, and they are different from the one we use.
}
For a definition, see \cite[B4.4.1]{Elephant1}.
However, easier to understand in our situation is to define the opfibrational
structure generically, as in \cite{FRV:Born},
which also discusses other ways in which fibrations and opfibrations relate to
the topos approaches to quantum theory.
Consider the two maps $X^i: X^\sier \to X$ where $i$ is either of the two principal
points $\bot,\top$ of $\sier$.
We have obvious maps $Y^\sier \to (X^i)^{*}Y$,
defined on $(x_\bot,y'_\bot,x_\top,y'_\top)$ by forgetting one of the $y'_i$s.
Then $p$ is an \textbf{opfibration} if the map $Y^\sier \to (X^\bot)^{*}Y$
has a left adjoint over $(X^\bot)^{*}Y$. We can write it in the form
\[
   (x_\bot, x_\top, y'_\bot) \mapsto (x_\bot, y'_\bot, x_\top, r_{x_\bot x_\top}(y'_\bot))\text{,}
\]
thus focusing attention on the \textbf{fibre map} $r_{x_\bot x_\top}: x_\bot^{*}Y \to x_\top^{*}Y$.

\begin{proposition}\label{prp:opfibspec}
  Let $p: Y \to X$ be a locale map that is a 2-categorical opfibration.
  Then $(x_1,y'_1) \spe (x_2,y'_2)$ in $Y$ iff $x_1 \spe x_2$ in $X$ and
  $r_{x_1 x_2}(y'_1) \spe y'_2$ in $x_2^{*}Y$.
\end{proposition}
\begin{proof}
  $\Rightarrow$: We have already noted that $x_1 \spe x_2$.
  Using the adjunction we find
  $(x_1,y'_1,x_2,r_{x_1 x_2}(y'_1))\spe(x_1,y'_1,x_2,y'_2)$
  and the conclusion follows.

  $\Leftarrow$: Using the fact that the adjoint maps into $Y^\sier$, we know
  $(x_1,y'_1) \spe (x_2,r_{x_1 x_2}(y'_1)) \spe (x_2,y'_2)$.
\end{proof}

There is a dual result if $p$ is a fibration, giving fibre maps contravariantly
$r_{x_\top x_\bot}: x_\top^{*}Y \to x_\bot^{*}Y$.
Then $(x_1,y'_1) \spe (x_2,y'_2)$ in $Y$ iff $x_1 \spe x_2$ in $X$ and
  $y'_1 \spe r_{x_2 x_1}(y'_2)$ in $x_1^{*}Y$.

If a locale property or construction is understood in terms of frames,
then its geometricity itself is likely to be non-obvious, because frames are non-geometric.
One potential way round this is to check preservation by the functors $f^{\#}$,
but this can be difficult without concrete knowledge of $f^{\#}$.
A more practical approach is usually to work with frame presentations by generators and relations.
A frame $\opens X$ that is presented by a set of generators $G$ and a set of relations $R$ will be denoted as $\opens
X\cong\mathbf{Fr}\langle G |R\rangle$. For the reader unfamiliar with frame presentations, see e.g. \cite[Chapter 4]{TVL} for an
introduction.
A general but primitive class of frame presentations is that of GRD-systems~\cite[Section 5]{PPExp},
but there are various more restricted kinds that are better adapted to particular problems.
Of whatever kind, the central property is that the relation between the presentation and the corresponding
bundle is geometric -- that is, preserved by bundle pullback.

\subsection{Toposes of presheaves and of sheaves} \label{sub:toppreshsh}

For the topos approach(es) to quantum theory we are in particular interested in toposes that
are either  functor categories or sheaves on a locale. In this subsection we consider geometricity for constructions on sets and locales in such toposes. As before, we are considering topos-valid constructions, so sets and their elements are understood in their generalized sense.

For functor categories, the fact that geometric constructions are preserved under the inverse image functor of
any geometric morphism entails the following lemma.

\begin{lemma}{(\cite[Corollary D1.2.14(i)]{Elephant1})} \label{lem:functormodel}
Let $\mathbb{T}$ be a geometric theory over a signature $\Sigma$ and let $\mathcal{C}$ be any small category.
A $\Sigma$-structure $M$ in the topos $[\mathcal{C},\Set]$ is a $\mathbb{T}$-model iff for every object $C\in\mathcal{C}_{0}$ the
$\Sigma$-structure $ev_{C}(M)$ is a $\mathbb{T}$-model in $\Set$.
Here $ev_{C}:[\mathcal{C},\Set]\to\Set$ denotes the functor that evaluates at the object $C$.
There is an isomorphism
\begin{equation*}
\Mod{\mathbb{T}}{[\mathcal{C},\Set]}\cong[\mathcal{C},\Mod{\mathbb{T}}{\Set}].
\end{equation*}
\end{lemma}

In this lemma the ``only if" part follows from the fact that $ev_{C}$ is the inverse image part of a geometric morphism.
The observation that we have an isomorphism of categories of models uses the fact that a homomorphism of
$\Sigma$-structures in $[\mathcal{C},\Set]$ can be identified with a natural transformation between the
$\Sigma$-structures, viewed as functors $\mathcal{C}\to\Str{\Sigma}{\Set}$.

\cite[Corollary D1.2.14(ii)]{Elephant1} 
shows that for a spatial locale $X$,
it suffices to check the fibrewise nature for fibres over the global points.
Let $\mathbb{T}$ be a geometric theory over a signature $\Sigma$.
Then a $\Sigma$-structure $M$ in the topos $\Sh(X)$ is a $\mathbb{T}$-model
iff, for each global point $x$, the fibre $x^{\ast}(M)$ is a $\mathbb{T}$-model in $\Set$.
Unfortunately, from a geometric point of view, spatiality is an uncommon property that often depends
on classical reasoning principles.
Lemma~\ref{lem:functormodel} is geometrically justified, because the Yoneda embedding $\mathcal{Y}$
already provides enough points of the form $\mathcal{Y}(C)$.

If the geometric theory $\mathbb{T}$ in Lemma~\ref{lem:functormodel} is classified by the topos $\mathcal{E}$, then the lemma
tells us that geometric morphisms from $[\mathcal{C},\Set]$ to $\mathcal{E}$ are equivalent to
$\mathcal{C}$-indexed diagrams of points of $\mathcal{E}$.
In fact, this can be used to show that $[\mathcal{C},\Set]$ is exponentiable as a topos,
with $\mathcal{E}^{[\mathcal{C},\Set]}$ classifying the theory of
$\mathcal{C}$-indexed diagrams of points of $\mathcal{E}$ (see e.g.~\cite{Johnstone/Joyal}).

We should also ask what are the geometric morphisms \emph{to} $[\mathcal{C},\Set]$,
i.e., its points.

\begin{definition}\label{defn:flatPresheaf}
  Let $\mathcal{C}$ be a small category, and $F:\mathcal{C}^{op}\to \Set$ a presheaf.
  $F$ is \textbf{flat} if it has the following properties
  \begin{enumerate}
  \item
    For some $C\in\mathcal{C}_{0}$, $F(C)$ is inhabited.
  \item
    If $x_{i}\in F(C_{i})$ ($i=1,2$), then there are morhisms $f_{i}:C_{i}\to D$
    in $\mathcal{C}$, and $y\in F(D)$, such that $x_{i}=F(f_{i})(y)$.
  \item
    If $f,g:C\to D$ in $\mathcal{C}$ and $y\in F(D)$ with
    $F(f)(y) = F(g)(y)$, there is is some morphism $h:D\to E$ such that $hf = hg$,
    and some $z\in F(E)$ such that $y = F(h)(z)$.
  \end{enumerate}
\end{definition}

The theory of flat presheaves over $\mathcal{C}$ is geometric, and its models are equivalent to the points of
$[\mathcal{C},\Set]$ (see e.g.~\cite{LocTopSp}).

The role of the flatness conditions becomes clearer if one considers the \textbf{Grothendieck construction},
which turns a presheaf $F$ into a category $\int F$, the so-called \textbf{category of elements} of $F$.
Its objects are pairs $(C,u)$ with $u\in F(C)$, and a morphism
from $(C,u)$ to $(D,v)$ is an arrow $f: C\to D$ such that $u = F(f)(v)$.
Then $F$ is flat iff $\int F$ is filtered, and in that case $F$ can be thought of as a filtered
diagram of representable presheaves $\mathcal{Y}(C)$ (which are themselves flat).
In fact, as point of $[\mathcal{C},\Set]$, $F$ is a filtered colimit of representables.

We shall mostly use this in the case where $\mathcal{C}$ is a poset $P$.

\begin{lemma}
For a poset $P$, flat presheaves $F:P^{op}\to\Set$ correspond up to isomorphism with ideals of $P$.
\end{lemma}

\begin{proof}
Let $F$ be a flat presheaf, and $p\in P$. By Definition~\ref{defn:flatPresheaf}(2) $F(p)$ contains at most one element. By this observation, we identify the set of objects of $\int F$ with a subset $I\subseteq P$, where $p\in I$ iff there exists an (automatically unique) element $x\in F(p)$, i.e., $(p,x)\in\int F$. Under the identification of $\int F$ with $I$, the arrows of $\int F$ translate to the order relation of $P$, restricted to $I$. By conditions (1) and (2) of Definition~\ref{defn:flatPresheaf}, $I$ is non-empty and upward directed. Condition (3) automatically holds for posets. The set $I$ is downward closed as $F$ is a presheaf. For a flat functor $F$, we have identified the category $\int F$ with an ideal $I$ of $P$. Conversely, to each ideal $I$ of $P$, we can associate a flat functor $F$, which has $I$ as its category of elements.
\end{proof}

Then the flat presheaves are simply the ideals of $P$, and we see that the functor topos $[P,\Set]$ is the topos of sheaves over the locale $\Idl P$
whose points are the ideals of $P$ and whose opens are the Scott opens of the ideal
completion, or, equivalently, the Alexandrov opens (up-closed subsets) of $P$.

A technique that we shall find useful in various places is a simple form of \textbf{iterated forcing}.
Suppose $\mathcal{D}$ is an internal category in $[\mathcal{C},\Set]$, hence a functor from
$\mathcal{C}$ to $\mathbf{Cat}$.
A topos $[\mathcal{D},[\mathcal{C},\Set]]$ of internal diagrams can be defined,
together with a geometric morphism to $[\mathcal{C},\Set]$.
A point of it is a pair $(F,G)$ where $F$ is flat presheaf over $\mathcal{C}$ and $G$ is
a flat presheaf over $\mathcal{D}(F)$,
i.e. $\mathop{\mathrm{colim}}_{(C,x)}\mathcal{C}(C)$, the filtered colimit of the diagram
corresponding to $F$.

It turns out (see e.g.~\cite[C2.5]{Elephant1})
that $[\mathcal{D},[\mathcal{C},\Set]]$ too is a presheaf topos, over the category
$\mathcal{C}\ltimes\mathcal{D}$ defined as follows.
Its objects are pairs $(C,D)$ with $C$ an object of $\mathcal{C}$ and
$D\in\mathcal{D}_{0}(C)$, where $\mathcal{D}_{0}:\mathcal{C}\to\Set$ is the object of objects of
$\mathcal{D}$.
An arrow $(f,g):(C,D)\to(C',D')$ in
$\mathcal{C}\ltimes\mathcal{D}$ is given by an arrow $f:C\to C'$ in $\mathcal{C}$
and an arrow $g\in\mathcal{D}_{1}(C')$,
$g:\mathcal{D}_{0}(f)(D)\to D'$.

As an example, consider the case where $\mathcal{D}$ is discrete, corresponding to an
object $U:\mathcal{C}\to \Set$ in $[\mathcal{C},\Set]$.
Then the geometric morphism
$[\mathcal{D},[\mathcal{C},\Set]]\to [\mathcal{C},\Set]$ is the local homeomorphism
corresponding to $U$.
We thus see that the external view of the internal discrete locale $U$ is also given
by a presheaf topos, on $C\ltimes U$.
Its objects are pairs $(C,u)$ ($C$ an object of $\mathcal{C}$, $u\in U(C)$),
and a morphism from $(C,u)$ to $(D,v)$ is a morphsm $f:C\to D$ such that
$v = U(f)(u)$.

\subsection{A result on exponentiability} \label{sub:expo}

This subsection presents a general result (Theorem~\ref{thm:expexp}) on exponentiability.
It is not specifically about geometricity,
but is related because geometricity would like to replace the non-geometric construction
of the frame $\opens X$ by the locale exponential $\sier^{X}$,
which has the same points ($\opens X$ is the set of points of $\sier^{X}$)
and is geometric.
However, the exponential exists only if $X$ is locally compact, and so we are going to be
interested in local compactness in order to define the elements of the frame geometrically.

In fact, the following known theorem~\cite{Hyl81} holds for locales in any elementary topos.

\begin{theorem} \label{thm:loccptloc}
Let $X$ be a locale, then the following are equivalent:
\begin{enumerate}
\item $X$ is locally compact.
\item The functor $(-)\times X:\Loc\to\Loc$ has a right adjoint $(-)^{X}$.
\item The exponential $\sier^{X}$ exists, where $\sier$ denotes the
Sierpi\'nski locale.
\end{enumerate}
\end{theorem}

We are going to build up to Theorem~\ref{thm:locexpexp}, stating that `locally perfect maps compose',
where $p: Y \to X$ is \textbf{locally perfect} if the corresponding internal locale
$\underline{Y}$ in $\Sh(X)$ is locally compact.%
\footnote{
  The terminology is debatable, and will depend on whether one calls $p$ \emph{perfect} or
  \emph{proper} for the situation where $\underline{Y}$ is compact.
  In the present paper we are following Johnstone's usage \cite{johnstoneloc},
  where a map $f:X\to Y$ between topological spaces is called locally perfect if the following condition holds.
  If $U\in\opens X$ and $x\in U$, then there exists an open neighborhood $V$ of $x$,
  an open neighborhood $N$ of $f(x)$, and a set $K\subseteq X$,
  such that $K\subseteq U\cap f^{-1}(N)$ and for each $b\in N$, the fibre $K_{b}$ is compact.
  He shows that if $f$ is locally perfect,
  then $f_*(X)$ is a locally compact locale in $\Sh(Y)$,
  and the converse holds if $Y$ satisfies the $T_D$ separation property.
}
The precise form of our result is that if $X$ is a locally compact locale,
and $p:Y\to X$ is locally perfect, then $Y$ is locally compact.
However, this can be relativized.
If also $q:Z\to Y$ is locally perfect, then by working in $\Sh(X)$ we find
that $pq$ is locally perfect.

We shall apply this to the spectrum of an internal C*-algebra to show that its external description is a locally compact locale.

Consider $p:Y\to X$ as above, and assume for a moment that the exponential $\sier^{Y}$ exists. A point of
$\sier^{Y}$ is equivalent to a map $Y\to\sier$, which corresponds to an open $U\in\opens Y$.  By assumption
$\underline{Y}$ is locally compact in $\Sh(X)$, so the exponential
$\underline{\sier}^{\underline{Y}}$ exists in $\Sh(X)$,
where $\underline{\sier}$ denotes the internal Sierpi\'nski locale.
Using the fact that $\Loc_{\Sh(X)}$ is equivalent to $\Loc/X$,
the locale $\underline{\sier}^{\underline{Y}}$ has an external description by a locale
map $q:\sier^{Y}_{X}\to X$, for some locale $\sier^{Y}_{X}$.
The external description of $\underline{\sier}$ is the projection
$\pi_{2}:\sier\times X\to X$.
As exponentiation of locales is geometric~\cite[Sec\ 10]{PPExp},
the fibre $q^*(\{x\})$ over a point $x$ in $X$ is given by $\sier^{Y_x}$,
where we write $Y_x$ for the fibre $x^* Y = p^{-1}(\{x\})$.

An open $U\in\opens Y$ and a point $x\in X$ give an open $U_{x}$ in the fibre $Y_x$,
as in the figure below.
This in turn is equivalent to a map $Y_x\to\sier$, which is a point of $q^*(\{x\})$.
This suggests that the global points of
$\sier^{Y}$ correspond exactly to the global sections of the bundle
$q:\sier^{Y}_{X}\to X$. That is, global points of $\sier^{Y}$
correspond to maps $\sigma: X\to\sier^{Y}_{X}$ such that $q\circ\sigma=\id_X$.

\begin{center}
 \includegraphics[height=50mm]{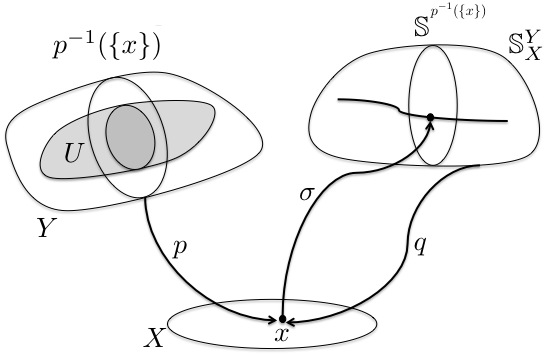}
\end{center}

The informal reasoning above can be made precise in order to prove that for a locally compact $X$, and $\underline{Y}$ a locally compact in $\Sh(X)$,
with external description $p:Y\to X$, the locale $Y$ is locally compact. We define locales by
providing their geometric theories (of generalized points) and we define
continuous maps as constructively described transformations of points of one locale to points
of another locale; see~\cite[Sec.~4.5]{LocTopSp}.

\begin{theorem}[locally perfect maps compose]\label{thm:locexpexp}
If $X$ is a locally compact locale, and $\underline{Y}$ a locally compact locale in $\Sh(X)$,
with external description $p:Y\to X$, then $Y$ is locally compact.
\end{theorem}
\begin{proof}
The locale $Y$ is the locale with (generalized) points the pairs
$(x,t)$ such that $x\in X,t\in Y_x$.
That is, we view the map $p$ as a bundle; see~\cite{Vickers:ContIsGeom} and the discussion of the bundle picture in Subsection~\ref{sub:geometricmath}.
As before, $\sier_X^Y$ denotes the external description of the internal locale
$\underline{\sier}^{\underline{Y}}$ in $\Sh(X)$.
Since the exponential is geometric \cite[Sec.\ 10.3]{PPExp},
the generalized points of $\sier_X^Y$ are the pairs $(x,w)$ such that
$x\in X, w\in \sier^{Y_x}$.
The internal evaluation map
$\underline{\ev}:
   \underline{\sier}^{\underline{Y}}\times\underline{Y}\to \underline{\sier}$,
part of the geometric structure of $\underline{\sier}^{\underline{Y}}$,
must then correspond to a map
$\ev: \sier_X^Y \times_X Y \to \sier$ given by $((x,w),y) \mapsto w(y)$.
We define $E$ as the locale with (generalized) points those $\sigma:X \to \sier_X^Y$
such that $q\circ \sigma=\id_X$.
Here local compactness of $X$ allows us to define the exponential $(\sier_X^Y)^X$,
and an equalizer captures the section condition to give a sublocale.
To define the map $\ev: E\times Y \to \sier$,
we first define the map $E\times Y \to \sier_X^Y\times_X Y$ by
$(\sigma,y)\mapsto (\sigma(py),y)$ and then compose it with the map
$\ev:\sier_X^Y\times_X Y \to \sier$.

To complete the proof, we require that if $g: Z \times Y \to \sier$,
then there is a unique $\tilde{g}:Z\to E$ such that
$g = \ev\circ(\tilde{g}\times \id_Y)$. This condition amounts to saying that for all $z$, $y$,
\[
  g(z,y) = w(y) \text{, where } \tilde{g}(z)(py) = (py,w)\text{.}
\]
Assuming this condition, consider $\tilde{g}(z)(x) = (x,w)$, for some $w$ in $\sier^{Y_x}$.
For all $y$ in $Y_x$ we must have $\tilde{g}(z)(py) = \tilde{g}(z)(x) = (py,w)$,
and by the condition $w(y) = g(z,y)$.
Hence $w$ is uniquely determined for each $x$, and so $\tilde{g}$ is unique.
Reversing the argument, and relying heavily of geometricity, we see that it leads to a
definition of $\tilde{g}$, given $g$.
\end{proof}

Theorem~\ref{thm:locexpexp} holds in a much more abstract categorical form.
This can be applied when the base $X$ is a non-localic topos that is exponentiable
in the category of toposes -- our prime example will be a topos of the form
$[\mathcal{C},\Set]$ where $\mathcal{C}$ is not a poset.

\begin{theorem}\label{thm:expexp}
  Let $\mathcal{C}$ be a category with finite limits, and let $X$ be an exponentiable object
  in $\mathcal{C}$.
  Let $p: Y \to X$ be an object $\underline{Y}$ of $\mathcal{C}/X$,
  let $Z$ be an object of $\mathcal{C}$, and suppose the exponential
  $\underline{Z}^{\underline{Y}}$ exists in $\mathcal{C}/X$.
  We shall write it as $q: Z^{Y}_{X}\to X$.
  Then $Z^{Y}$ exists in $\mathcal{C}$.
\end{theorem}
\begin{proof}
By the considerations above Theorem~\ref{thm:locexpexp}, we arrive at the
following candidate for the exponential $Z^{Y}$.  Take the equalizer
\[
\xymatrix{E\  \ar@{^{(}->}[r]^{eq} & (Z^{Y}_{X})^{X}
\ar@/^/[r]^{q^{X}} \ar@/_/[r]_{\ulcorner X\urcorner\circ !} & X^{X}},
\]
where $\ulcorner X\urcorner:1\to X^{X}$ denotes the transpose of the identity arrow of $X$. Note that the
exponentials $(Z^{Y}_{X})^{X}$ and $X^{X}$
exist in $\mathcal{C}$ by exponentiability of $X$.
Also note that the global points of $E$ are exactly the global sections of $q$.
Next, we need to find a suitable evaluation map $ev: E\times Y\to Z$.
For the definition of $ev$ we will make use of the internal evaluation arrow
$\underline{Z}^{\underline{Y}}\times\underline{Y}\to\underline{Z}$.
Externally this gives the following commuting triangle:
\[
\xymatrix{ Z^{Y}_{X}\times_{X} Y \ar[rr]^{ev} \ar[dr] & &
Z\times X \ar[dl]^{\pi_{2}} \\
& X & } \]
With some abuse of notation, we denote the map $\pi_{1}\circ ev:Z^{Y}_{X}\times_{X}Y\to Z$ again by
$ev$. For the next step in defining the evaluation map, the diagram given below is commutative by the definition of $E$.

\[ \xymatrix{ E\times Y \ar[ddd]_{\pi_{2}} \ar[r]^{E\times p} & E\times X  \ar@/_3pc/[dddrrrr]_{\pi_{2}}
\ar[rr]^{eq\times X} & & (Z^{Y}_{X})^{X}\times X \ar[rr]^{ev} \ar[ld]_{!\times X} \ar[ddr]^{q^{X}\times X} & &
Z^{Y}_{X} \ar[ddd]^{q} \\
& & 1\times X \ar[drr]_{\ulcorner X\urcorner\times X} & & & \\
& & & & X^{X}\times X \ar[dr]^{ev} &\\
Y \ar[rrrrr]_{p} & & & & & X} \]

The evaluation maps and exponentials in this diagram exist because of the exponentiability of $X$.
By the universal property of pullbacks this diagrams yields
an arrow $E\times Y\to Z^{Y}_{X}\times_{X} Y$.
Taking the composition with $ev:Z^{Y}_{X}\times_{X} Y\to Z$ coming from the internal
exponential $\underline{Z}^{\underline{Y}}$ gives the desired
evaluation map $ev: E\times Y\to Z$.

It remains to check that this map satisfies the desired universal property.
Maps $Z\times Y\to Z$ correspond bijectively with maps $(Z\times
X)\times_{X} Y\to Z\times X$ over $X$. By the existence of the internal
exponent $\underline{Z}^{\underline{Y}}$ the latter maps correspond
bijectively with maps $Z\times X\to Z^{Y}_{X}$ over $X$.
Using exponentiability of $X$ maps $Z\times X\to Z^{Y}_{X}$ correspond bijectively
with maps $Z\to(Z^{Y}_{X})^{X}$. The maps  $Z\times
X\to Z^{Y}_{X}$ that are maps over $X$ precisely correspond to the maps
$Z\to(Z^{Y}_{X})^{X}$ that factor through $E$. This proves that $E$ is
indeed an exponential $Z^{Y}$.
\end{proof}

In particular it follows from the theorem that if $\underline{Y}$ is exponentiable in $\mathcal{C}/X$,
then $Y$ is exponentiable in $\mathcal{C}$. Theorem~\ref{thm:locexpexp} is the special case where $\mathcal{C}$ is the category of
locales.

\section{Toposes and Quantum Theory} \label{sec:quantum}

The topos approaches to quantum theory were inspired by the work of
Butterfield and Isham \cite{butterfieldisham1, butterfieldisham3, butterfieldisham2, butterfieldisham4}. The original ``BI'' approach,
or presheaf approach (as it is formulated using a topos of presheaves) was subsequently developed
by D\"oring and Isham~\cite{DoeringIsham:WhatThingTTFP}.
In the present paper we are particularly interested in the different formulation of~\cite{qtopos},
which we call the copresheaf approach.
What these two approaches share is their use, for a given quantum system described by a C*-algebra,
of a topos whose internal logic embodies the idea of fixing a generic \textbf{context},
or classical perspective on that system.
Thus in the internal logic, the system has a classical phase space.
Our aim here is to show how this internal phase space can be represented
externally by a bundle, whose fibres are the phase spaces for individual contexts.

In the approaches cited, a unital C*-algebra\footnote{In the presheaf approach the smaller class of von Neumann algebras are typically considered, rather than arbitrary unital C*-algebras.} $A$ is studied through the poset $\mathcal{C}(A)$
of commutative unital C*-subalgebras of $A$, ordered by inclusion,
and using a topos of functors from $\cC{A}$ to $\Set$.

At this point the presheaf and copresheaf approaches diverge, using contravariant and covariant functors respectively.
We shall focus now on the copresheaf approach with its topos $[\cC{A},\Set]$.
To study the operator algebra $A$ from the perspectives of its commutative C*-subalgebras,
it is replaced by the covariant functor
$\ua:\mathcal{C}(A)\to\Set$, $\ua(C)=C$, where the arrows are mapped to inclusions.
In the topos $[\mathcal{C}(A),\Set]$,
$\ua$ becomes an internal unital \emph{commutative} C*-algebra.

From the internal perspective of the topos, the quantum-mechanical observables,
now in the form of the commutative C*-algebra $\ua$,
look more like the observables of a classical physical theory.
Crucial for this is the topos-valid version of Gelfand duality,
described in the work of Banaschewski and Mulvey \cite{banaschewskimulvey00a,banaschewskimulvey00b,banaschewskimulvey06}
as a duality between the category of unital commutative C*-algebras and
the category of compact, completely regular locales.
A more explicit and fully constructive description of this Gelfand duality is given in \cite{coquand05,CoquandSpitters:cstar}.
By this duality, $\ua$ is isomorphic to the $C^{\ast}$-algebra of continuous,
complex-valued functions on a certain compact, completely regular locale $\uS$,
the spectrum of $\ua$, which we think of as a phase space.
For further discussion of the ideas of the topos approach, such as the treatment of states, see
\cite{qtopos,Bohrification_ql,Bohrification,CaspersHLS:IntQLnLS,FRV:Born}. We also mention exciting recent work which reconstructs the
Jordan algebra structure from the spectral
object~\cite{Hamhalter:iso-order,HamhalterTurilova,HardingDoering,Doering:2012a,Doering:2012b}.

In this section we apply the geometricity ideas of the previous section to the topos approach described
in \cite{qtopos}.
In particular, we concentrate on the description of the spectrum of internal commutative
unital C*-algebras in toposes.
The first subsection examines the geometricity of parts of the definition of C*-algebra.
This is a central issue, since the C*-algebra structure itself is not geometric,
so it must be replaced by a geometric structure that,
while not being the whole C*-algebra,
nonetheless supports a geometric construction of the spectrum.

The second subsection examines the geometricity of the construction of the spectrum.
The third subsection is concerned with finding an explicit
description of the spectrum of the algebra $\ua$ in the particular topos $[\mathcal{C}(A),\Set]$,
which is central to the topos approach.
The fourth subsection generalizes this description to spectra of commutative unital C*-algebras in functor categories in general. In the fifth subsection we consider extending the copresheaf approach to algebraic quantum field theory. Throughout the emphasis will be on the role of geometric logic.

\subsection{C*-algebras in toposes} \label{sub:CStarAlg}

In a topos $\mathcal{E}$, a C*-algebra $A$ (always assumed unital here) --
\begin{enumerate}
\item
  is an associative, unital algebra over $\mathbb{C}$ with an involution ${*}$, which is anti-multiplicative ($(ab)^{\ast}=b^{\ast}a^{\ast}$) and conjugate linear ($(za+b)^{\ast}=\bar{z}a^{\ast}+b^{\ast}$, where $z\in\mathbb{C}$),
\item
  has a norm $\lVert - \rVert$, which is submultiplicative and satisfies the axiom $\lVert aa^{*}\rVert = \lVert a \rVert^{2}$,
  and
\item
  is complete with respect to the norm.
\end{enumerate}

However, the definition is not geometric, for three reasons.

First, with regard to the algebraic structure (1) above, already $\mathbb{C}$ --
as \emph{set} of complex numbers rather than a locale -- is not a geometric construction.
For a truly geometric account, a C*-algebra would have to be a locale rather than a set.
In~\cite{banaschewskimulvey06}, a C*-algebra is defined as an algebra over the Gaussian
numbers $\scal$, upon which completeness then allows the action of $\scal$ to be extended
to $\mathbb{C}$.
Apart from this, the algebraic part is straightforward.
It consists of arrows
$\underline{+}, \underline{\cdot}:\ua\times\ua\to\ua$
for addition and multiplication,
an arrow $\underline{\ast}:\ua\to\ua$ for the involution,
an arrow
$\underline{\scal}\times\ua\to\ua$ for scalar multiplication,
and constants $0,1:\underline{1}\to\ua$ for the unit and zero element.
These arrows are to render all desired diagrams commutative.
Note that we made use of the geometricity of $\scal$.
In the language of Subsection~\ref{sub:geometriclogic},
in a formal geometric theory we could declare a sort $k$,
add structure and axioms to force it to be isomorphic to $\mathbb{Q}\times\mathbb{Q}$,
and then define the appropriate operations as ring with involution.

Second, the norm is not geometric.
We shall develop a geometric theory of ``commutative G*-algebra'',
more general than unital commutative C*-algebras, but expressing enough of the structure
to define the spectrum.
For the norm, we shall ignore the condition
$\lVert a \rVert = 0 \rightarrow a=0$, as it is not geometric%
\footnote{
  When the norm is expressed as a subobject of $A\times\mathbb{Q}^{+}$,
  the potential non-geometricity of the axiom is visible in $\lVert a \rVert = 0$ as a universal quantifier over $\mathbb{Q}^{+}$.
}:
so we have only a semi-norm.
As explained in~\cite{LocCompA}, the semi-norm as described in~\cite{banaschewskimulvey06} with a
binary relation $N\subseteq A\times\mathbb{Q}^{+}$ can be understood as a map
$\lVert\cdot\rVert:A\to\upperR$ taking its values in the \emph{upper reals}.
Then $(a,q)\in N$ if $\lVert a\rVert <q$. The third reason, connected to the first one, is that the completeness of $A$ with respect to the norm is not geometric.

What we seek, therefore, is a geometric notion that generalizes commutative C*-algebras,
and on which we can still, and geometrically, calculate the spectrum.

There is a geometric core to the definition of C*-algebra,
in the notion of \textbf{semi-normed pre-C*-algebra} --
that is, a *-algebra over $\mathbb{Q}[i]$, but dropping completeness and weakening the norm
to a semi-norm.
However, that runs into problems because at a certain point in constructing the spectrum
we need to know the the order on the self-adjoints $A_{sa}$, or, alternatively, its positive cone.
In a C*-algebra a self-adjoint is (non-strictly) positive iff it is a square,
but without completeness we cannot guarantee the vital property that the sum of squares is still a square -- for example, 2 might not be a square.

Our way round this is to use the preordered archimedean rings of~\cite{coquand05}.

\begin{definition}
  A commutative $\mathbb{Q}$-algebra $R$ is called \textbf{preordered} if it has a \textbf{positive cone},
  i.e., a subset $P$ that contains all squares and is closed under addition and multiplication by $\mathbb{Q}^{+}$.
  The preorder is then given by $a\leq b$ if $b-a\in P$.

  The preorder $R$ is \textbf{archimedean} if, in addition, for each $a\in R$ there is some $r\in \mathbb{Q}$
  such that $a\leq r$.

  A \textbf{commutative G*-algebra} is a commutative%
    \footnote{
      In this ad hoc naming, we write `G' for `geometric'. We have not attempted to define non-commutative G*-algebras,
      since it is not so easy to order the self-adjoints when they don't commute.
    }
  $\scal$-*-algebra $A$ for which the self-adjoint part
  $A_{sa}$ is a preordered archimedean ring.
\end{definition}

Since the definition is geometric, there is a canonical notion of homomorphism between two
commutative G*-algebras: a function that preserves the *-algebra structure and positivity.

\subsection{The Gelfand spectrum and normal lattices} \label{sub:Gelf}\label{sub:bohr}

If $A$ is a commutative C*-algebra in a topos, then its spectrum $\Sigma_{A}$ is the space of the continuous *-algebra
homomorphisms $x:A\to\mathbb{C}$.
The classical theory then says that $A$ is isomorphic to the complex algebra of continuous maps from $\Sigma_{A}$
to $\mathbb{C}$,
and its self-adjoint part $A_{sa}$ is isomorphic to the real algebra of continuous maps from $\Sigma_{A}$
to $\mathbb{R}$.
The topology on $\Sigma_{A}$ is the weak-* topology, and another way to say this is that
a subbasis of opens is provided by the sets of the form $\{x\mid x(a)>0\}$ for $a\in A$ self-adjoint.

This suggests that if we want a geometric description of the points of the spectrum, we should use
the elements $a\in A_{sa}$ to form propositional symbols (let us say $D(a)$),
and add axioms to say that in this system $D(a)$ behaves like $\{x\mid x(a)>0\}$.
This was done in~\cite{banaschewskimulvey06}.

Coquand~\cite{coquand05} defines, for any preordered archimedean ring, the spectrum (which he calls the maximal spectrum) in a point-free
way.

If $A$ is a commutative C*-algebra, $A$ is a commutative G*-algebra by restricting $\mathbb{C}$ to $\scal$ and defining its positive cone to be the set of
all squares. This is closed under addition and multiplication by positive rationals.
Also, the archimedean property follows from the existence of the norm.
Then~\cite{CoquandSpitters:cstar} its point-free Gelfand spectrum is isomorphic to the spectrum constructed in \cite{coquand05}.

As mentioned before, a point-free approach will give axioms characterizing the behaviour of
formal symbols $D(a)$ ($a\in A_{sa}$) with intended meaning $\{x\mid x(a)>0\}$. The axioms in~\cite{coquand05} are:
\begin{subequations}\label{equ:SigmaA1}
\begin{align}
  D(a) \wedge D(-a) &\vdash \bot             \label{equ:SigmaA1a}\\
  D(a) &\vdash \bot \text{ if } a\leq 0      \label{equ:SigmaA1b}\\
  D(a+b) &\vdash D(a) \vee D(b)              \label{equ:SigmaA1c}\\
  \top &\vdash D(1)                          \label{equ:SigmaA1d}\\
  D(ab) &\dashv\vdash
    (D(a)\wedge D(b)) \vee (D(-a)\wedge D(-b)) \label{equ:SigmaA1e}
\end{align}
\end{subequations}
\begin{equation}\label{equ:SigmaA2}
  D(a) \vdash \bigvee_{0<r\in\mathbb{Q}}D(a-r).
\end{equation}

These are all intuitively clear in terms of the intended meaning.
For example, (\ref{equ:SigmaA1a}) says that for
no $x$ can we have both $x(a)>0$ and $x(-a) = -x(a) > 0$,
while (\ref{equ:SigmaA2}) says that if
$x(a)>0$, then $x(a)>r$ for some $r>0$.

The fact that these axioms are enough is not at all obvious, and takes up some substantial,
non-trivial calculations~\cite{coquand05}.
The upshot is that each point $x$ of the spectrum can be described
geometrically by the set of those elements $a\in A_{sa}$ for which $x(a)>0$.
In fact, we could take it (but we shall modify this view) that that set \emph{is} the point $x$.
Such a set must conform with the axioms.
For example, it cannot contain both $a$ and $-a$.

We do not wish to recap those substantial calculations,
but there is an important structural part of the development of which we shall make considerable use.
The six axioms above (\ref{equ:SigmaA1},\ref{equ:SigmaA2}) are a propositional geometric theory,
with propositional symbols $D(a)$ indexed by elements
$a$ of $A_{sa}$. As mentioned in Section~\ref{sub:geometriclogic},
it can be used to present a frame $F_A$, with the $D(a)$ as generators and the axioms as relations.
Then the points of the spectrum correspond to completely prime filters of the frame.
However, the first group of axioms (\ref{equ:SigmaA1}) do not use infinitary disjunctions,
and so could be taken as presenting a finitary distributive lattice.
\begin{definition}
 Let $A$ be a commutative G*-algebra.
 We write $L_{A}$ for the distributive lattice presented by generators
 $D(a)$ ($a\in A_{sa}$) subject to the above relations (\ref{equ:SigmaA1}).
 (In~\cite{coquand05}, $L_A$ is referred to as $\mathrm{Spec}_{r}(A_{sa})$.)
\end{definition}

Unlike the frame $F_A$, the lattice $L_A$ is constructed geometrically from $A$.
This is not noted explicitly in~\cite{coquand05},
but is a consequence of the way it is presented by generators and relations
because geometric constructions include free algebras, and generating and factoring out congruences.
The techniques are as described in \cite{PHLCC}.
It is, in fact, possible to give a description that is more concrete than that of the general universal algebra,
though we shall not need it here.
\cite{EntSys} shows how to construct, geometrically, a free distributive lattice as a quotient
of a double finite powerset, in this case $\mathcal{FF}A_{sa}$.
\cite{coquand05} then gives an explicit concrete description of when, in $L_{A}$, we have
\[
   D(a_1) \wedge \cdots \wedge (a_{n}) \leq D(b_1) \vee \cdots \vee D(b_{m})\text{.}
\]

It follows that points $x$ of the spectrum can also be represented as subsets of $x\subseteq L_{A}$.
Each subset corresponds to a map from $L_{A}$ to the subobject classifier $\Omega$,
and respecting the axioms (\ref{equ:SigmaA1}) then amounts to saying that this map preserves
meets and joins (all finitary), in other words that $x$ is a \textbf{prime filter} of $L_{A}$.
It is a filter if it is up-closed and closed under finite meets,
and it is prime if whenever it contains a finite join, then it also contains one of the elements joined.

At this point we can thus describe the points of the spectrum of $A$ geometrically as prime filters $x$
of $L_{A}$, not arbitrary ones, but those respecting axiom (\ref{equ:SigmaA2}), in other words such that if $D(a)\in x$ then $D(a-r)\in x$
for some $0<r\in\mathbb{Q}$. This is slightly awkward because, although we have reduced much to the
lattice, we still have to refer explicitly to the elements of $A$. The next stage will remove
this awkwardness, and at the same time give access to a general geometric
treatment of compact regular locales.

\begin{definition}
A distributive lattice $L$ is said to be \textbf{normal} if whenever $a\vee b = \top$
then there are $x$ and $y$ such that $a\vee y = x\vee b = \top$ and $x\wedge y=\bot$.
Defining $a'\wi a$ ($a'$ \textbf{well inside} $a$) if there is $y$ such that $a\vee y = \top$ and $a'\wedge y = \bot$,
then another way to express normality is that if $a\vee b = \top$, then there is some $a'\wi a$
with $a'\vee b = \top$.
We also write $\dd a$ for the set $\{a'\mid a'\wi a\}$.

If $L$ is a normal distributive lattice, then a prime filter $x\subseteq L$ is \textbf{regular}
if whenever $a\in x$, then $a'\in x$ for some $a'\wi a$.
\end{definition}

The theory of regular ideals of a normal distributive lattice $L$ is geometric.
We could describe this by a propositional theory along the lines of that used earlier for $\Sigma_A$,
but to illustrate the use of geometric mathematics we present a predicate theory that describes
the regular ideals directly.
It has a predicate symbol $x\subseteq L$, and axioms as follows.
\begin{subequations}\label{equ:RSpec}
\begin{align}
  \top & \vdash x(\top)                         \label{equ:RSpeca}\\
  x(a) & \wedge x(b) \vdash_{ab} x(a \wedge b)  \label{equ:RSpecb}\\
  x(\bot) & \vdash \bot                         \label{equ:RSpecc}\\
  x(a \vee b) & \vdash_{ab} x(a) \vee x(b)      \label{equ:RSpecd}\\
  x(a) & \vdash_{a} (\exists a')(a'\wi a \wedge x(a')) \label{equ:RSpece}
\end{align}
\end{subequations}
(Note that the logical symbols $\top,\wedge,\bot,\vee$ are overloaded here,
denoting both operators in the lattice $L$ and logical connectives.
The two usages are syntactically quite different, though clearly the
axioms set up a semantic connection between them.)
Axioms (\ref{equ:RSpeca},\ref{equ:RSpecb}) say that $x$ is a filter,
(\ref{equ:RSpecc},\ref{equ:RSpecd}) that it is prime,
and (\ref{equ:RSpece}) that it is regular.

This is presented as a predicate theory, and so, in principle, has a classifying
topos that might not be localic.
However, it \emph{is} localic, and this is evident from the fact that no new sorts
are declared in the signature.
The cautious reader can write down an equivalent propositional theory explicitly --
this process is discussed in \cite{LocTopSp} for the Dedekind reals.

We write $\RSpec L$ (the \textbf{regular spectrum} of $L$) for the locale just defined,
whose points are the regular prime filters of $L$.

Coquand~\cite{coquand05} proved
(i) $L_{A}$ is normal, and
(ii) axiom (\ref{equ:SigmaA2}) is equivalent to the regularity axiom $a\vdash \bigvee_{a'\wi a} a'$.
This is important as it implies that the spectrum $\Sigma_{A}$ is isomorphic to $\RSpec L_{A}$,
the regular spectrum of $L_{A}$.

This completes our geometric construction of the spectrum from the C*-algebra.
We replaced the non-geometric commutative C*-algebra by the more general and geometric commutative G*-algebra.
The spectrum of this G*-algebra $A$ is constructed geometrically by first constructing the normal distributive lattice $L_{A}$ and subsequently constructing its regular spectrum.

If $f:A\to B$ is a homomorphism of commutative G*-algebras, then geometricity ensures that it gives
a lattice homomorphism $L_{f}: L_{A}\to L_{B}$, and (contravariantly) a map
$\Sigma_{f}:\Sigma_{B} \to \Sigma_{A}$.
If $y$ is a regular prime filter of $L_{B}$, then $\Sigma_{f}(y)$ is its
inverse image $L_{f}^{-1}(y)$.
If $f$ is an inclusion, we shall generally write
$L_{AB}$ for $L_{f}$ and $\rho_{BA}$ for $\Sigma_{f}$.

In the next two lemmas we shall consider formal expressions
$\phi(x_1,\ldots,x_n) = \phi(x_i)_{1}^{n}$
built using finite meets and finite joins from generators
$D(x_i)$, where $x_i$ is a variable ranging over $A_{sa}$.
From \cite{coquand05} we know that $\phi(a_i -r)_{1}^{n} \wi \phi(a_i)_{1}^{n}$
if $0<r\in\mathbb{Q}$.
\cite[Corollary 1.7]{coquand05} also shows that if $\phi$ is a finite join
of generators, and $\phi(a_i)_{1}^{n}=1$ in $L_A$ for some elements
$a_i \in A_{sa}$, then $\phi(a_i -r)_{1}^{n} = 1$ for some $0<r\in\mathbb{Q}$.
This clearly extends to the case where $\phi$ is a finite meet of finite joins
of generators, and hence to arbitrary distributive lattice expressions.

\begin{lemma}\label{lemma:WellInsideShrink}
 Let $A$ be a commutative G*-algebra, and suppose
 $v \wi u = \phi(a_i)_{1}^{n}$ in $L_A$.
 Then there is some $0<r\in\mathbb{Q}$ such that
 $v \leq \phi(a_i -r)_{1}^{n}$.
\end{lemma}
\begin{proof}
 We have $u\vee w = 1$ and $v\wedge w = 0$ for some $w$,
 and -- extending the list of $a_i$s as necessary --
 we can write $w = \psi(a_i)_{1}^{n}$.
 By the previous discussion we can find $0<r\in\mathbb{Q}$ such that
 $\phi(a_i -r)_{1}^{n}\vee \psi(a_i -r)_{1}^{n} = 1$,
 and then, knowing that $\psi(a_i -r)_{1}^{n}\leq w$,
 we conclude that $v\leq \phi(a_i -r)_{1}^{n}$.
\end{proof}

The following lemma will be important when we come to describe the external spectrum.

\begin{lemma}
\label{lemma:WellInsideInc}
 Let $f:A\to B$ be a homomorphism of G*-algebras, and suppose $u\in L_{A}$
 and $v\wi L_{f}(u)$ in $L_{B}$.
 Then there is some $u'\wi u$ such that $v\leq L_{f}(u')$.
\end{lemma}
\begin{proof}
 Writing $u = \phi(a_i)_{1}^{n}$, we find that
 $v\wi \phi(f(a_i))_{1}^{n}$.
 Applying Lemma~\ref{lemma:WellInsideShrink}, we can find $0<r\in\mathbb{Q}$
 such that
 $v\leq \phi(f(a_i)-r)_{1}^{n} = L_f(u')$
 where $u' = \phi(a_i -r)_{1}^{n}\wi u$.
\end{proof}

The opens of $\RSpec L$ are described in \cite{qtopos} as the regular ideals of $L$, i.e. those ideals
$I$ such that if $\dd a\subseteq I$ then $a\in I$.
This follows from the coverage theorem,
a general result of topos-valid locale theory, but unfortunately it is not geometric.
The problem lies in the regularity condition, which amounts to
\[
   (\forall a')(a'\wi a \rightarrow I(a')) \vdash_{a} I(a)
\]
where the left-hand side is not a geometric formula.

In order to gain access to geometric methods, we replace the regular ideals by the
``rounded $\wi$-ideals''.
They differ in the way they use elements of $L$ to represent opens:
a regular ideal comprises those elements of $L$ that are less than the open,
whereas the $\wi$-ideal comprises those that are well inside.

The concept of rounded $\wi$-ideal is geometric,
and so there is a locale $\RIdl (L,\wi)$ (or, for short, $\RIdl L$)%
\footnote{Note that the `R' in $\RIdl$ stands for \emph{rounded},
not \emph{regular}.
}
whose points are the rounded $\wi$-ideals.
This much follows \cite{Infosys} from the simple fact that $(L,\wi)$ is a \emph{continuous information system} --- that is,
$\wi$ is an idempotent relation:
transitive ($\mathord{\wi}\circ\mathord{\wi} \subseteq \mathord{\wi}$) and
interpolative ($\mathord{\wi}\circ\mathord{\wi} \supseteq \mathord{\wi}$).

What is important in our situation is that the points of the locale $\RIdl L$ form a frame,
and indeed the frame of opens for $\RSpec L$.
This is expressed succinctly in Theorem~\ref{thm:RSpecLocComp},
that $\RIdl L$ is the locale exponential $\sier^{\RSpec L}$,
which allows us to use Theorem~\ref{thm:loccptloc}.

Given any idempotent relation $<$ on a set $X$,
\cite{Infosys} in effect defines a \textbf{rounded $<$-ideal} to be a model of the
geometric theory with one predicate symbol $I\subseteq X$ and axioms
\begin{align*}
  a' < a \wedge I(a) & \vdash_{a' a} I(a')     \\
  \top & \vdash (\exists a) I(a)             \\
  I(a') \wedge I(a'') & \vdash_{a' a''} (\exists a) (a' < a \wedge a'' < a \wedge I(a))
\end{align*}
The corresponding locale is $\RIdl(X,<)$,%
\footnote{
  \cite{Infosys} calls it $\Idl(X,<)$, but we want to stress the roundedness.
}
and its opens are the rounded upsets of X under $<$.

The proof of the following theorem is instructive.
Although the theorem is stated entirely for locales,
the proof is fairly simple if one digresses into non-localic toposes.
Note that the discussion in Section~\ref{sub:geometricmath} of opfibrations
holds equally for maps (geometric morphisms) between toposes.
A map from $\sier$ to a topos $\mathcal{E}$ is a pair of points of
$\mathcal{E}$, together with a specialization morphism between them.

\begin{theorem}\label{thm:RIdlopfib}
  Let $X$ be a locale, and $(P,<)$ a continuous information system in $\Sh(X)$.
  Let $p: Y \to X$ be the external description of the internal locale $\RIdl(P,<)$.
  Then $p$ is an opfibration.
\end{theorem}
\begin{proof}
  Let $\mathcal{E}$ be the classifying topos for the theory of continuous information
  systems $(Q,<)$,
  and let $q:\mathcal{F}\to\mathcal{E}$ be the bundle in which the fibre over $(Q,<)$
  is the rounded ideal completion of $(Q,<)$.
  Hence $\mathcal{F}$ classifies triples $(Q,<,J)$ where $(Q,<)$ is a continuous information
  system and $J$ is a rounded ideal for it.
  We first show that $q$ is an opfibration.

  Let $f:(Q_\bot,<)\to (Q_\top,<)$ be a homomorphism of continuous information systems,
  in other words, a function $f:Q_\bot\to Q_\top$ that preserves $<$.
  By \cite[Proposition~2.10]{LocCompB}, we obtain a map
  $\RIdl(f): \RIdl(Q_\bot,<) \to \RIdl(Q_\top,<)$,
  mapping $J \mapsto \dd f(J)$.
  We have $\RIdl(f)(J)\spe J'$ iff $f$ extends to a homomorphism (necessarily unique) from
  $(Q_\bot,<,J)$ to $(Q_\top,<,J')$,
  and from this it follows that the map
  $((Q_\bot,<),(Q_\top,<),J) \mapsto (Q_\bot,<,J,Q_\top,<,\RIdl(f)(J))$
  provides the left adjoint required for an opfibration; see Proposition~\ref{prp:opfibspec}.

  Returning to the situation as in the statement,
  by the definition of classifying topos the continuous information system in $\Sh(X)$
  gives a map $(P,<): X \to \mathcal{E}$,
  and $p$ is the pullback of $q$ along it.
  From the geometricity of the definition of opfibration, we see that the property is
  preserved under pullback.
\end{proof}

To see that $\wi$ is an idempotent relation (which is well known),
first note that in any distributive lattice $L$ we have that if
$b'\leq a' \wi a \leq b$ then $b'\wi b$, and if $a' \wi a$ then $a' \leq a$.
From this is follows that $\wi$ is transitive.
If moreover $L$ is normal, then $\wi$ is interpolative.
For if $a' \wi a$ with $y$ as in the definition, then by normality $a''\vee y = \top$
for some $a'' \wi a$, and then also $a' \wi a''$.
Another useful fact is that if $a_i \wi a$ ($i=1,2$), then $a_1 \vee a_2 \wi a$.
We find also that the rounded $\wi$-ideals of $L$ are exactly the ordinary ideals $I$
(with respect to $\leq$) that are rounded in the sense that $I = \dd I$.

\begin{proposition}
  Let $L$ be a normal distributive lattice.
  Then there is a bijection between its regular ideals and its $\wi$-ideals.
\end{proposition}
\begin{proof}
  First, if $J$ is a regular ideal then $\dd J$ is a $\wi$-ideal.
  The only mildly non-obvious part of this is that if $a'_i\wi a$ ($i=1,2$),
  using $y_{i}$ as in the definition, then $a'_{1}\vee a'_{2} \wi a$ using $y_{1}\wedge y_{2}$.

  Next, if $I$ is a $\wi$-ideal then we write
  \[
    r\langle I \rangle = \{a\in L \mid \dd a \subseteq I\}\text{.}
  \]
  To show this is an ideal, suppose $\dd a_{i} \subseteq I$ and $b\wi a_1 \vee a_2$
  with $y$ as in the definition.
  Then we can find $a_{i}' \wi a_i$ with $a'_1 \vee a'_2 \vee y = \top$
  and it follows that $b \wi a'_1 \vee a'_2 \in I$.
  It is regular, because if $\dd a \subseteq r\langle I \rangle$ then $\dd\dd a \subseteq I$;
  but $\dd\dd a = \dd a$ by interpolativity of $\wi$,
  so $a\in r\langle I\rangle$.
  Indeed, it is the smallest regular ideal containing $I$.

  It is now easy to show this gives a bijection, with $r\langle\dd J\rangle = J$ and
  $\dd r\langle I \rangle = I$ (using $\dd I = I$).
\end{proof}

We now prove our central result in this section,
which gives a geometric account of the opens of $\RSpec L$.
One can compare it with Stone's representation theorem for Boolean algebras $B$
(see \cite{StoneSp}), where the frame of opens for the spectrum of $B$ is the ideal
completion of $B$. In fact, our result generalizes that, since every Boolean algebra
is normal, with $\wi$ coinciding with $\leq$.

\begin{theorem}
\label{thm:RSpecLocComp}
  If $L$ is a normal distributive lattice, then $\RIdl L \cong \sier^{\RSpec L}$.

  The evaluation map $\RIdl L \times \RSpec L \to \sier$ takes $(I,x)$ to the top point $\top$
  whenever the ideal $I$ meets the filter $x$.
\end{theorem}
\begin{proof}
We use \cite[Theorem 12.7]{PPExp},
which states that if $X$ and $W$ are locales for which $\sier^{X}$ exists
and is homeomorphic to the double powerlocale $\mathbb{P}W$,
then $\sier^{W}$ exists and is homeomorphic to $X$.
In our case we take $X$ to be $\RIdl L$,
which is exponentiable with its opens the rounded upsets of $L$,
and $W$ to be $\RSpec L$, so it remains to calculate the double powerlocale $\mathbb{P}W$
and show that its points are again the rounded upsets of $L$.

We use the Double Coverage Theorem \cite[Theorem 7]{UniCharPP} to calculate its points.
As explained in \cite{CompLocFT}, the calculation is an analogue of one for the lower powerlocale
that is directly derived from the usual Coverage Theorem \cite{StoneSp}.
Given a site in the form of a meet-semilattice $S$ with a meet-stable coverage $\cov$,
the corresponding locale has for its points the filters $F$ of $S$ that ``split'' $\cov$
in the sense that if $a \cov U$ and $a\in F$ then $F$ meets $U$.
Then its lower powerlocale has for its points the upsets of $S$
(thus we drop the requirement that meets should be respected) that split $\cov$.
For the double powerlocale we require a \emph{DL-site},
in which $S$ is a distributive lattice and the coverage is stable for both meets and joins:
the points for the locale are the prime filters that split $\cov$,
and for the double powerlocale they are the upsets that split $\cov$.

To apply this theorem we must show that the covers $a \cov \dd a$ in $L$ are meet- and join-stable.
Meet-stability means that $a\wedge b$ is covered by $\{a'\wedge b\mid a'\wi a\}$.
It is \emph{not} true that if $a'\wi a$ then $a'\wedge b\wi a\wedge b$.
However, knowing that $a\wedge b\cov\dd(a\wedge b)$,
we argue that if $c\wi a\wedge b$ then $c\wi a' \wi a$ for some $a'$, and so $c \leq a'\wedge b$.
Join-stability is similar, but with joins.

It is clear that an upset splits $\cov$ iff it is rounded with respect to $\wi$,
and the result follows. The final part derives from the definition of the evaluation map
for $\RIdl L$.
\end{proof}

Now that the notion of commutative G*-algebra has been settled, and the geometricity of the construction of its spectrum has been demonstrated, the following theorem is immediate.

\begin{theorem} \label{thm:bundle}
Let $\ua$ be a unital commutative C*-algebra in a topos $\mathcal{E}$ and let $\uS_{\ua}$ be the Gelfand spectrum of $\ua$.
Let $\Sh_{\mathcal{E}}(\uS_{\ua})\to\mathcal{E}$ be the unique localic geometric morphism corresponding to the locale
$\uS_{\ua}$, and let $f:\mathcal{F}\to\mathcal{E}$ be any geometric morphism. Consider the commutative G*-algebra $f^{\ast}(\ua)$,
and let $\uS_{f^{\ast}(\ua)}$
be its spectrum in $\mathcal{F}$.
Then we have a pullback square
\[ \xymatrix{ \Sh_{\mathcal{F}}(\uS_{f^{\ast}(\ua)}) \ar[d] \ar[r] &
\Sh_{\mathcal{E}}(\uS_{\ua}) \ar[d]\\
\mathcal{F} \ar[r]_{f} & \mathcal{E} } \]
\end{theorem}

We finish the section by observing that Theorem~\ref{thm:RSpecLocComp} and the results
in~\cite{Coquand/Spitters:integrals-valuations,Vickers:Riesz} provide entirely analoguous results for the geometric theory of integrals on a
G*-algebra. Following~\cite{Bohrification} these internal integrals correspond to quasi-states externally.

\subsection{Computation of the spectrum} \label{sub:calcspec1}

In this subsection we fix a unital C*-algebra $A$ and calculate, in various forms,
the spectrum as it arises in the copresheaf approach.

To summarize the notation, $\mathcal{C}(A)$ is the poset of unital commutative C*-subalgebras of $A$, partially ordered by inclusion,
and the topos $[\mathcal{C}(A),\Set]$ is the category of copresheaves on $\mathcal{C}(A)$,
equivalent to the sheaf topos $\Sh(\Idl\mathcal{C}(A))$.
$\ua$ is the tautological copresheaf mapping each context $C$ to itself,
and mapping each arrow $D\to C$ in $\mathcal{C}(A)$ to the inclusion $D\hookrightarrow C$.
It is a unital commutative C*-algebra, internal to the topos $[\mathcal{C}(A),\Set]$
and leads -- internally -- to a normal distributive lattice $\underline{L}_{\ua}$
and a Gelfand spectrum $\uS_{\ua} \cong \underline{\RIdl}(\underline{L}_{\ua})$.
This compact regular locale is of interest to the copresheaf model, as it internally plays the role of a phase space.

We seek the external representions of the Gelfand spectrum and other internal locales as bundles,
which we shall typically denote by removing underlinings --
sometimes in an \emph{ad hoc} way.
The external representation of $\uS_{\ua}$ will be $p: \Sigma_{\ua}\to\Idl\mathcal{C}(A)$.

First we calculate the locale $\Sigma_{\ua}$.
We characterize both its points (Theorem~\ref{thm:RSpecExternalPoints}),
in line with the geometric approach,
and its opens (Theorem~\ref{prp:extspec}).
We find it convenient to use local compactness to describe the opens geometrically as points of
the exponential $\sier^{\Sigma_{\ua}}$,
and this is easily translated into a description of the frame $\opens\Sigma_{\ua}$.

The explicit description of the opens was previously given Subsection 2.2 of \cite{Wolters}.
Although the proof given there has the advantage of not using any
advanced topos-theoretic methods, it has some disadvantages too.
The proof hides the role of geometric reasoning.
As we have seen in the previous subsection,
it is because of geometricity that the spectrum is so closely related to the
spectra of the commutative C*-subalgebras
(or, for reader familiar with the work of Butterfield and Isham,
why the spectrum is so closely related to the spectral presheaf of their approach).
Another disadvantage of the proof in \cite{Wolters} is that it is not clear how it can be generalized
when the topos $[\mathcal{C}(A),\Set]$ is replaced by a different topos.
This point is also related to the geometricity being hidden.

\begin{theorem}
\label{thm:RSpecExternalPoints}
  The points of $\Sigma_{\ua}$ can be geometrically described as the pairs
  $(I,x)$ where $I$ is an ideal of $\cC{A}$ and $x$ is a subset of $\coprod_{C\in \cC{A}}L_{C}$
  satisfying the following properties:
  \begin{enumerate}
  \item
    If $(C,a)\in x$ then $C\in I$.
  \item
    If $C\subseteq D\in I$ then $(C,a)\in x$ iff $(D,L_{CD}(a))\in x$.
  \item
    If $C\in I$ then $\{a\in L_{C}\mid(C,a)\in x\}$ is a regular prime filter in $L_{C}$.
  \end{enumerate}
\end{theorem}
\begin{proof}
Recall from Subsection~\ref{sub:geometricmath} that a point of $\Sigma_{\ua}$ is equivalent to a pair $(I,x)$ where $I$ is an ideal of $\cC{A}$ (i.e. a point of the locale $\Idl\mathcal{C}(A)$), and $x$ is a point of
$I^{*}\Sigma_{\ua}$, in other words, a regular prime filter of $L_{I}$.
Now $I$ is the filtered colimit (actually here a directed join) of the principal ideals $\downset C$ for
$C\in I$, and geometric constructions preserve filtered colimits.
It follows that
\[
  L_{I} \cong \mathop{\mathrm{colim}}_{C\in I} L_{C}\text{.}
\]
Each element of $L_{I}$ is the image of some $(C,a)$, so to specify the filter $x$ of $L_{I}$
it suffices to specify $x\subseteq \coprod_{C\in \cC{A}}L_{C}$.
Then condition (1) says that $x$ is in the colimit over $C\in I$, and (2) expresses the fact that equality in
the filtered colimit derives from $(C,a)=(D,L_{CD}(a))$.
The final condition expresses the regular prime filter property. However, some care is needed with regularity. Actually, regularity says that
if $(C,a)\in x$ then there is some $D\in I$ with $C\subseteq D$ and some $b\wi L_{CD}(a)$
in $L_{D}$ such that $(D,b)\in x$.
Lemma~\ref{lemma:WellInsideInc} says that in this case there is some $a'\in L_{C}$ with $a'\wi a$ and
$b \leq L_{CD}(a')$, and we then have that $(C,a')\in x$.
\end{proof}

Note how this geometric description provides a subbase for the topology (a set of generators for the frame)
as the pairs $(C,a)$, where $C\in \cC{A}$ and $a\in L_{C}$. The point $(I,x)$ is in this open iff $C\in I$
and $(C,a)\in x$.

We now use local compactness of $\uS_{\ua}$ to give an explicit description in geometric form of the internal frame, as $\underline{\sier}^{\uS_{\ua}}$.
We shall write $\sier^{\uS_{\ua}}$ for its external description.

\begin{theorem}\label{thm:RSpecInternalOpens}
  The points of $\sier^{\uS_{\ua}}$ are pairs $(I,U)$ where $I$ is an ideal of $\cC{A}$
  and $U$ is a subset of $\coprod_{C\in\cC{A}}L_{C}$ satisfying the following properties.
  \begin{enumerate}
  \item
    If $(C,a)\in U$ then $C\in I$.
  \item
    If $C\subseteq D\in I$ then $(C,a)\in U$ iff $(D,L_{CD}(a))\in U$.
  \item
    If $C\in I$ then $\{a\in L_{C}\mid(C,a)\in U\}$ is an ideal in $L_{C}$.
  \item
    If $(C,a)\in U$ then there is some $(D,b)\in U$ with $C\subseteq D$ and $L_{CD}(a)\wi b$.
  \end{enumerate}
\end{theorem}
\begin{proof}
  The first part is an application of Theorem~\ref{thm:RSpecLocComp}.
  The rest is done in exactly the same way as Theorem~\ref{thm:RSpecExternalPoints}, except we have
  to take care in expressing the fact that $L_{I}$ is rounded under $\wi$, because
  Lemma~\ref{lemma:WellInsideInc} does not apply.
  Note that in the case where $I$ is a principal ideal $\downset D$,
  $U$ is equivalent to a $\wi$-ideal of $L_{D}$ and hence to an open of $\Sigma_{D}$.
\end{proof}

We can now give an explicit description of the frame $\opens\Sigma_{\ua}$.

\begin{theorem} \label{prp:extspec}
$\Sigma_{\ua}$ is locally compact. The points of $\sier^{\Sigma_{\ua}}$ are the
$\cC{A}$-indexed families $U$, where each $U_{C}$ is a $\wi$-ideal of $L_{C}$,
and if $C\subseteq D$, then $L_{CD}(U_{C})\subseteq U_{D}$.
\end{theorem}
\begin{proof}
By Theorem~\ref{thm:expexp} we know that $\Sigma_{\ua}$ is locally compact, and that
$\sier^{\Sigma_{\ua}}$ has the sections of $\sier^{\uS_{\ua}}$ for its points.
Since a section is a map from $\Idl \cC{A}$, we can use Lemma~\ref{lem:functormodel}
to see that the sections are as described in the statement.
\end{proof}

In terms of frames, it is now immediate that $\opens \Sigma_{\ua}$ is isomorphic to the subframe of $\prod_{C\in\mathcal{C}(A)}\opens \Sigma_{C}$
comprising those elements $U$ such that for any
$C\subseteq D$, $\rho^*_{DC}(U_{C})\subseteq U_{D}$.

Unlike the case for the external $\sier^{\Sigma_{\ua}}$,
it is a non-trivial calculation to calculate the internal frame
$\underline{\opens}(\uS_{\ua}) = \underline{\pt}(\underline{\sier}^{\uS_{\ua}})$
We should like to emphasize that this non-trivial calculation is usually unnecessary.
However much one might like to think that the internal locale ``is'' the internal frame,
it is usually better to identify the locale with its external description.

\begin{theorem}\label{thm:RSpecInternalFrame}
  The internal frame $\underline{\opens}(\uS_{\ua})$ is given
  as a copresheaf by
  \[
    \underline{\opens}(\uS_{\ua}) =
      \{U\in\prod_{C\subseteq D}\opens\Sigma_{D} \mid
         \text{ if } C \subseteq D_{1}\subseteq D_{2} \text{ then }
         \rho^{\ast}_{D_{2}D_{1}}(U_{D_{1}}) \subseteq U_{D_{2}}\}.
  \]
\end{theorem}
\begin{proof}
  By Yoneda's Lemma,
  \[
    \underline{\opens}(\uS_{\ua})(C)
         \cong [\cC{A},\Set](\mathcal{Y}(C),\underline{\opens}(\uS_{\ua}),
  \]
  where $\mathcal{Y}:\cC{A}^{op}\to [\cC{A},\Set]$ is the Yoneda embedding.
  Now, internally, $\underline{\opens}(\uS_{\ua})$ is the set of points of
  $\underline{\sier}^{\uS_{\ua}}$, in other words its discrete coreflection.
  It follows that morphisms $\mathcal{Y}(C)\to \underline{\opens}(\uS_{\ua})$
  are equivalent to locale maps
  \[
    \mathcal{Y}(C)\to \underline{\sier}^{\uS_{\ua}}
        \cong \underline{\RIdl}(\underline{L}_{\ua})\text{,}
  \]
  and we can analyse these by their external representation.
  Externally, the local homeomorphism for $\mathcal{Y}(C)$ has bundle space $\Idl(\upset C)$.
  This can be proved by iterated forcing -- see Section~\ref{sub:toppreshsh}.
  Hence we seek maps $\Idl(\upset C) \to \sier^{\uS_{\ua}}$ over $\Idl \cC{A}$ and,
  by Lemma~\ref{lem:functormodel},
  the required maps are the same as monotone families $U_{D}$ ($C\subseteq D$), with each $U_{D}$
  in the fibre of $\sier^{\uS_{\ua}}$ over $\downset D$, in other words a rounded ideal of $L_D$.
  ``Monotone'' is with respect to the given order on $\cC{A}$ and the specialization order on
  $\sier^{\uS_{\ua}}$.
  Using Proposition~\ref{prp:opfibspec} and Theorem~\ref{thm:RIdlopfib},
  we see that if $C \subseteq D_{1}\subseteq D_{2}$,
  then we require $\RIdl(L_{D_1 D_2})(U_{D_{1}}) \subseteq U_{D_{2}}$.
  Now suppose $I \subseteq L_{D_1}$ is a $\wi$-ideal,
  and $y \subseteq L_{D_2}$ a regular prime filter.
  Then the point $\rho^{\ast}_{D_{2}D_{1}}(y) = L_{D_1 D_2}^{-1}(y)$ is in the open $I$ iff,
  as subsets of $L_{D_1}$, they \textbf{meet} -- they have inhabited intersection.
  Clearly this is equivalent to $L_{D_1 D_2}(I)$ meeting $y$,
  which in turn is equivalent to $y$ being in $\RIdl(L_{D_1 D_2})(I) = \dd L_{D_1 D_2}(I)$.
  It follows that $\RIdl(L_{D_1 D_2}) = \rho^{\ast}_{D_{2}D_{1}}$,
  and we have the description in the statement.
\end{proof}

\subsection{C*-algebras in functor categories} \label{sec:functor}

In this subsection we generalize Theorem~\ref{prp:extspec} to unital
commutative C*-algebras in toposes that are functor categories.
Subsequently, in Subsection~\ref{examples}, we use this result when we explore examples of functor categories
(other than $[\mathcal{C}(A),\Set]$, or $[\mathcal{C}(A)^{\text{op}},\Set]$)
which may be of interest to the topos approaches to quantum theory.
As a first step, we use presheaf semantics to identify all C*-algebras in functor categories.
The reader unfamiliar with presheaf semantics may want to consult \cite[Chapter VI]{maclanemoerdijk92}.

Let $\mathcal{C}$ be any small category. Below we prove the following:

\begin{proposition} \label{prp:internal}
The object $\ua$ (with additional structure $\underline{+},\underline{\cdot}, \underline{\ast},\underline{0}$)
is a C*-algebra in the topos $[\mathcal{C},\Set]$ iff it is given by a functor
$\ua:\mathcal{C}\to\mathbf{CStar}$, where $\mathbf{CStar}$ is the category of C*-algebras and
$\ast$-homomorphisms in $\Set$. The C*-algebra $\ua$ is commutative iff each $\ua(C)$ is
commutative. The algebra $\ua$ is unital iff every $\ua(C)$ is unital and for each $f:C\to D$, the
$\ast$-homomorphism $\ua(f):\ua(C)\to\ua(D)$ preserves the unit.
\end{proposition}

\begin{proof}
It follows from Lemma~\ref{lem:functormodel} and the discussion
in Subsection~\ref{sub:bohr} that a semi-normed $\ast$-algebra over $\underline{\mathbb{C}}$ in $[\mathcal{C},\Set]$
is equivalent to a functor $\ua:\mathcal{C}\to\Set$, such that each $\ua(C)$ is a semi-normed $\ast$-algebra over $\mathbb{C}$, and,
for each arrow $f:D\to C$ in $\mathcal{C}$, the function $\ua(f):\ua(D)\to\ua(C)$ is a $\ast$-homomorphism such that $\Vert\ua(f)(a)\Vert_{D}\leq\Vert a\Vert_{C}$.
We used $\Vert\cdot\Vert_{C}$ to denote the semi-norm on $\ua(C)$. The internal semi-norm $\underline{N}$ is submultiplicative and satisfies the C*-property
iff each semi-norm $\Vert\cdot\Vert_{C}$ is submultiplicative and satisfies the C*-property $\Vert a^{\ast}a\Vert_{C}=\Vert a\Vert_{C}^{2}$.

Recall that the semi-norm $\underline{N}$ of $\ua$ is defined as a subobject of $\ua\times\underline{\mathbb{Q}^{+}}$.
The internal semi-norm is connected to the external semi-norms by the identities
\begin{equation} \label{equ:norm}
\underline{N}(C)=\{(a,q)\in\ua(C)\times\mathbb{Q}^{+}\mid\Vert a\Vert_{C}<q\},
\end{equation}
\begin{equation} \label{equ:normie}
\Vert\cdot\Vert_{C}:\ua(C)\to\mathbb{R}^{+}_{0},\ \ \Vert a\Vert_{C}=\text{inf}\{q\in\mathbb{Q}^{+}\mid(a,q)\in\underline{N}(C)\}.
\end{equation}
Note that the fact that $\ast$-homomorphisms are contractions, in the sense that
$\Vert\ua(f)(a)\Vert_{C}\leq\Vert a\Vert_{C}$,
precisely states that $\underline{N}$ defined by (\ref{equ:norm}) is a well-defined subobject of $\ua\times\underline{\mathbb{Q}^{+}}$. The semi-norm $\underline{N}$ is a norm iff it satisfies the axiom
\begin{equation} \label{equ:norm0}
(\forall q\in\underline{\mathbb{Q}^{+}}\ \ (a,q)\in\underline{N})\Rightarrow(a=\underline{0}).
\end{equation}
By the rules of presheaf semantics, externally, this axiom translates to: for each $C\in\mathcal{C}$ the semi-norm $\Vert\cdot\Vert_{C}$ is a norm.

Completeness can be checked in the same way as in \cite{qtopos}, because the axiom of dependent choice works in any presheaf topos. For
completeness, we thus need to check the axiom
\begin{align*}
\forall f\in\ua^{\underline{\mathbb{N}}} & (\ (\forall n\in\underline{\mathbb{N}}\ \forall
m\in\underline{\mathbb{N}}\ (f(n)-f(m),2^{-n}+2^{-m})\in\underline{N}) \nonumber \\
& \ \Rightarrow(\exists a\in\ua\ \forall n\in\underline{\mathbb{N}} (a-f(n),2^{-n})\in\underline{N})\ ).
\end{align*}
Note that for any object $C\in\mathcal{C}$, the elements of $\ua^{\underline{\mathbb{N}}}(C)$ correspond
exactly to sequences $(a_{n})_{n\in\mathbb{N}}$ in $\ua(C)$. By presheaf semantics, the axiom for completeness
holds iff for every object $C\in\mathcal{C}$ and any sequence $(a_{n})_{n\in\mathbb{N}}$ in $\ua(C)$, if
\begin{equation*}
C\Vdash  (\forall n\in\underline{\mathbb{N}}\ \forall m\in\underline{\mathbb{N}}\
(a_{n}-a_{m},2^{-n}+2^{-m})\in\underline{N}),
\end{equation*}
then
\begin{equation*}
C\Vdash (\exists a\in\ua\ \forall n\in\underline{\mathbb{N}}\  (a-a_{n},2^{-n})\in\underline{N}).
\end{equation*}
This can be simplified by repeated use of presheaf semantics, and the identity
\begin{equation*}
\Vert\ua(f)(a_{n})-\ua(f)(a_{m})\Vert_{D}=\Vert\ua(f)(a_{n}-a_{m})\Vert_{D}\leq\Vert
a_{n}-a_{m}\Vert_{C},
\end{equation*}
where $f:C\to D$ is any arrow. In the end, the axiom of completeness simplifies to the statement that given an object
$C\in\mathcal{C}$ and any sequence $(a_{n})_{n\in\mathbb{N}}$ in $\ua(C)$ such that for any pair
$n,m\in\mathbb{N}$ we have $(a_{n}-a_{m},2^{-n}+2^{-m})\in\underline{N}(C)$, there exists an element
$a\in\ua(C)$ such that for every $n\in\mathbb{N}$, $(a-a_{n},2^{-n})\in\underline{N}(C)$. By definition of
$\underline{N}$ this simply states that each $\ua(C)$ is complete with respect to the norm
$\Vert\cdot\Vert_{C}$. This completes the proof that C*-algebras in $[\mathcal{C},\Set]$ are equivalent to functors $\mathcal{C}\to\mathbf{CStar}$.
\end{proof}

Proposition~\ref{prp:internal} makes it easy to calculate the spectrum $\uS_{\ua}$ of any
commutative unital C*-algebra $\ua$ in any functor category,
using the methods of Subsection~\ref{sub:calcspec1}.
The first step is the construction of the distributive lattice $\underline{L}_{\ua}$,
which we will simply denote as $\underline{L}$.
This is done as in Subsection~\ref{sub:Gelf}.
The construction is geometric, based on the G*-algebra structure of $\ua$,
and so from the discussion in Section~\ref{sec:geo} we can derive $\underline{L}$:
\begin{equation} \label{equ:latt}
\underline{L}:\mathcal{C}\to\Set,\ \ \underline{L}(C)=L_{\ua(C)},
\end{equation}
\begin{equation*}
\underline{L}(f) = L_{\ua(f)}:L_{\ua(C)}\to L_{\ua(D)}\text{.}
\end{equation*}

As in Subsection~\ref{sub:calcspec1}, we can now describe the spectrum $\uS$
in various ways.
To start with, it is described externally as a localic geometric morphism
$\Sigma\to [\mathcal{C},\Set]$.
We must keep in mind that if $\mathcal{C}$ is not a poset,
then neither $\Sigma_{\ua}$ nor $[\mathcal{C},\Set]$ is a localic topos:
they classify non-propositional theories and are \emph{generalized} spaces in
the sense of Grothendieck.
However, the bundle $\Sigma_{\ua} \to[\mathcal{C},\Set]$ is still a \textbf{localic bundle},
because it is a localic geometric morphism~\cite[A4.6.1]{Elephant1}.
Its fibres are locales.
This all follows because it arises from an internal locale in $[\mathcal{C},\Set]$.

To make clear the analogy with the localic case, we shall use the same notation for
a locale as for its topos of sheaves.
For example, 1 will often denote the topos of sets, and $\sier$ the topos
$\Sh\sier$, whose objects are functions.
In practice this does not cause confusion, but one should refrain from asking whether
the locale actually ``is'' the frame or the topos.
Rather, one should think that a locale is described equally well by its frame of opens
or its topos of sheaves.
Note that geometric morphisms between localic toposes are equivalent to maps between
the locales, and we shall frequently refer to geometric morphisms in general as
\emph{maps} between the toposes.

For the rest of this subsection we shall let $\mathcal{C}$ be a small category,
and let $\ua$ be a C*-algebra in $[\mathcal{C},\Set]$.
We write $p: \Sigma_{\ua}\to [\mathcal{C},\Set]$ for the localic bundle
corresponding to the internal spectrum $\uS_{\ua}$.

\begin{theorem}
\label{thm:RSpecExternalPointsFunctor}
  (Analogue of Theorem~\ref{thm:RSpecExternalPoints}.)
  The points of $\Sigma_{\ua}$ can be geometrically described as the pairs
  $(F,x)$ where $F$ is a flat presheaf over $\mathcal{C}$ and
  $x$ is a subset of
  $\coprod_{(C,u)\in (\int F)_{0}}L_{C}$
  satisfying the following properties:
  \begin{enumerate}
  \item
    If $f: (C,u)\to (D,v)$ in $\int F$ then $(C,u,a)\in x$ iff $(D,v,L_{f}(a))\in x$.
  \item
    If $(C,u)$ is an object in $\int F$ then $\{a\in L_{C}\mid(C,u,a)\in x\}$
    is a regular prime filter in $L_{C}$.
  \end{enumerate}
\end{theorem}
\begin{proof}
The proof is essentially the same as that of Theorem~\ref{thm:RSpecExternalPoints},
once one has taken on the fact (Subsection~\ref{sub:toppreshsh}, just below Definition~\ref{defn:flatPresheaf}) that the points
of $[\mathcal{C},\Set]$ are the flat presheaves over $\mathcal{C}$ and that each
is a filtered colimit of representables $\mathcal{Y}(C)$.
\end{proof}

\begin{theorem}\label{thm:RSpecInternalOpensFunctor}
  (Analogue of Theorem~\ref{thm:RSpecInternalOpens}.)
  The frame of opens of $\uS_{\ua}$ has internal localic form
  $\underline{\sier}^{\uS_{\ua}} = \underline{\RIdl}( \underline{L}_{\ua})$.
  Externally, the points of the bundle locale $\sier ^{\uS_{\ua}}$ are pairs $(F,U)$,
  where $F$ is a flat presheaf of
  $\mathcal{C}$ and $U$ is a subset of $\coprod_{(C,u)\in(\int F)_{0}}L_{C}$
  satisfying the following properties:
  \begin{enumerate}
  \item
    If $f:(C,u)\to (D,v)$ in $\int F$, then $(C,u,a)\in U$ iff $(D,v,L_{f}(a))\in U$.
  \item
    If $(C,u)$ is an object in $\int F$, then $\{a\in L_{C}\mid(C,u,a)\in U\}$
    is an ideal in $L_{C}$.
  \item
    If $(C,u,a)\in U$, then there is some $f: (C,u)\to (D,v)$ in $\int F$
    and some $(D,v,b)\in U$ with $L_{f}(a)\wi b$.
  \end{enumerate}
\end{theorem}
\begin{proof}
  Essentially the same as for Theorem~\ref{thm:RSpecInternalOpens}.
\end{proof}

At this point in Subsection~\ref{sub:calcspec1} we moved on to
giving an explicit description of the frame $\opens\Sigma_{\ua}$.
Here we must be more careful, since in general $\Sigma_{\ua}$ is a non-localic topos.
Any topos map (geometric morphism) has a hyperconnected-localic factorization~\cite[A4.6.1]{Elephant1},
and by applying this to a map $\mathcal{E}\to 1$ we see that $\mathcal{E}$ has a
\emph{localic reflection} $\mathcal{L}(\mathcal{E})$.
Its opens are the maps $\mathcal{E}\to\sier$.
In other words, since $\sier$ classifies subsingletons,
the opens of $\mathcal{L}(\mathcal{E})$ are the subobjects of 1 in $\mathcal{E}$.
If $\mathcal{E}$ is exponentiable as a topos (as is the case for $[\mathcal{C},\Set]$),
then the opens of $\mathcal{L}(\mathcal{E})$ are the points of $\sier^{\mathcal{E}}$
and can be calculated immediately by the method of Theorem~\ref{prp:extspec}.

\begin{theorem} \label{prp:extspecFunctor}
(Analogue of Theorem~\ref{prp:extspec}.)
  The opens of $\Sigma_{\ua}$
  (equivalently: the opens of the localic reflection $\mathcal{L}(\Sigma_{\ua})$)
  can be geometrically described as the elements
  $U\in\prod_{C\in\mathcal{C}_{0}}\Sigma_{\ua(C)}$
  such that if $f:C\to D$ is a morphism in $\mathcal{C}$, then
  $\Sigma^{\ast}_{\ua(f)}(U_{C})\leq U_{D}$.
\end{theorem}
\begin{proof}
  Using exponentiability as in Theorem~\ref{prp:extspec}, we can calculate the
  opens as the points of $\sier^{\Sigma_{\ua}}$ and thence as the sections of
  $\sier^{\uS_{\ua}}\to [\mathcal{C},\Set]$.
  Sections are described using Lemma~\ref{lem:functormodel}.
  For each object $C$ of $\mathcal{C}$ we require a $\wi$-ideal of $L_{\ua(C)}$,
  hence an element of $\opens\Sigma_{\ua(C)}$.
  For each morphism $f:C\to D$ we require a homomorphism of points of
  $\sier^{\uS_{\ua}}$ over $f$, and the homomorphism property comes down to
  the condition $\Sigma^{\ast}_{\ua(f)}(U_{C})\leq U_{D}$.
\end{proof}

For completeness, we sketch how one might describe the full topos $\Sh(\Sigma_{\ua})$),
i.e.\ the sheaves of $\Sigma_{\ua}$ rather than just its opens (which are not enough, and only describe
the localic reflection). Again one would use exponentiability, but this time to
calculate $\oc^{\Sigma}$ instead of $\sier^{\Sigma}$.
Here, $\oc$ denotes the \textbf{object classifier},
classifying the geometric theory with one sort and no other ingredients.
Its models are just sets (carrying the single sort),
so the maps $\Sigma\to \oc$ are the sheaves of $\Sigma$ (by which we just mean the
objects of the topos, even in this non-localic case).
\cite{ViglasThesis} shows how to give a geometric description of the sheaves over any
stably compact locale when presented by a strong proximity lattice,
and as a special case works for a compact regular locale when presented by a
normal distributive lattice $L$.
The sheaves are, first of all, finitary sheaves over the distributive lattice $L$,
i.e. those presheaves $F$ that satisfy finite instances of the sheaf pasting condition.
In addition, they must satisfy a continuity condition that
$F(a)\cong \mathop{\mathrm{colim}}_{a\wi a'} F(a')$.

We now proceed by analogy with Theorem~\ref{thm:RSpecInternalOpensFunctor},
but replacing $\sier$ by $\oc$.
Internally in $[\mathcal{C},\Set]$ we have a geometric theory describing the sheaves
over the internal locale $\uS_{\ua}$.
This corresponds externally to a bundle (\emph{not} localic)
$\oc^{\uS_{\ua}}\to [\mathcal{C},\Set]$.
We can now describe the points of the external topos $\oc^{\uS_{\ua}}$
as pairs $(F,G)$ where $F$ is a flat presheaf over $\mathcal{C}$ and $G$
is a sheaf (presheaf with finitary pasting and continuity) over
$\mathop{\mathrm{colim}}_{(C,u)\in(\int F)_{0}}L_{\ua(C)}$.
After this, the sheaves of $\Sigma$ can be described as the sections of
$\oc^{\uS_{\ua}}\to [\mathcal{C},\Set]$.

Finally in this subsection, we calculate the internal frame of $\uS_{\ua}$.

\begin{theorem}\label{thm:RSpecInternalFrameFunctor}
  (Analogue of Theorem~\ref{thm:RSpecInternalFrame}.)
  The internal frame $\underline{\opens}(\uS_{\ua})$ is given
  as an object of $[\mathcal{C},\Set]$ by
  \[
    \underline{\opens}(\uS_{\ua})(C) =
      \{U\in\prod_{f:C\to D}\opens\Sigma_{D} \mid
         \text{ if } C \stackrel{f}{\rightarrow} D_{1}\stackrel{g}{\rightarrow} D_{2} \text{ then }
         \rho^{\ast}_{g}(U_{f}) \leq U_{gf}\}
  \]
\end{theorem}
\begin{proof}
  Just as in Theorem~\ref{thm:RSpecInternalFrame}, we find that
  $\underline{\opens}(\uS_{\ua})(C)$ is the set of maps (geometric morphisms)
  from $\mathcal{Y}(C)$ to $\RIdl L_{\ua}$ over $[\mathcal{C},\Set]$.
  We can calculate the external form of the discrete locale $\mathcal{Y}(C)$
  using iterated forcing (Subsection~\ref{sub:toppreshsh}).
  Its topos is $[\mathcal{C}\ltimes\mathcal{Y}(C),\Set]$,
  where now $\mathcal{Y}(C)$ denotes the discrete internal category.
  From the definition of $\mathcal{Y}$ we see that the objects of
  $\mathcal{C}\ltimes\mathcal{Y}(C)$ are the pairs $(D,f)$ with $f: C\to D$ in
  $\mathcal{C}$,
  and a morphism from $(D_{1},f_{1})$ to $(D_{2},f_{2})$ is a morphism
  $g:D_{1}\to D_{2}$ in $\mathcal{C}$ such that $f_{2} = gf_{1}$.

  Now combining Lemma~\ref{lem:functormodel} with
  Theorem~\ref{thm:RSpecInternalOpensFunctor},
  we see that the maps we are looking for are functors $U$ from
  $\mathcal{C}\ltimes\mathcal{Y}(C)$ to the models as described in
  Theorem~\ref{thm:RSpecInternalOpensFunctor}.
  For each object $f:C\to D$, its image $U_{f}$ must be such that the flat presheaf $F$
  as described in the Theorem is representable (for $D$).
  Its representability allows us to combine conditions (2) and (3)
  to say that we have a $\wi$-ideal in $L_{\ua(D)}$,
  hence an open of $\Sigma_{\ua(D)}$.
  Now for each morphism $g:D\to D'$, with $f'=gf$, we need a homomorphism
  $U_{f}\to U_{f'}$ over $g$.
  In other words, if $a\in U_{f}$ (as $\wi$-ideal of $L_{\ua(D)}$),
  then $L_{\ua(g)}(a)\in U_{f'}$.
  When the $\wi$-ideals of $L_{\ua(D)}$ are viewed as opens of
  $\Sigma_{\ua(D)}$, this is equivalent to saying that
  $U_{f}\leq \rho^{\ast}_{\ua(g)}(U_{f'})$.
  We now see that the internal frame is as described.
\end{proof}

\subsection{Algebraic quantum field theory} \label{examples}

The copresheaf approach is based on algebraic quantum theory in the sense that in this approach quantum theory is described using abstract
C*-algebras. In this section we seek to extend the copresheaf approach, as already suggested in~\cite{qtopos}, to the Haag--Kastler
formalism,
which is an algebraic approach to quantum field theory. Introductions to the Haag--Kastler formalism, or algebraic quantum field theory
(AQFT), can be found in~\cite{araki,Haag:LQP}. In this formalism (where, for the sake of simplicity we consider Minkowski spacetime
$\mathcal{M}$) the physical content of a quantum field theory is described by a net of C*-algebras $O\to\frk{A}(O)$, where $O$ ranges over
certain (open connected causally complete) regions of spacetime. This means that we associate to each region $O$ of spacetime of interest, a
C*-algebra $\frk{A}(O)$. We think of the self-adjoint elements of $\frk{A}(O)$ as the observables that can be measured in the region $O$.
With this in mind, we can make the assumption that if $O_{1}\subseteq O_{2}$, then $\frk{A}(O_{1})\subseteq\frk{A}(O_{2})$. If
$\mathcal{K}(\mathcal{M})$ denotes the set of the spacetime regions of interest, partially ordered by inclusion, then a net of C*-algebras
defines a covariant functor $\frk{A}:\mathcal{K}(\mathcal{M})\to\mathbf{CStar}$. We assume that the algebras $\frk{A}(O)$ are unital for
convenience.

By Proposition~\ref{prp:internal} an AQFT is a C*-algebra $\frk{A}$ internal to $[\mathcal{K}(\mathcal{M}),\Set]$. Note that $\frk{A}$ is
in general not commutative. As for the copresheaf approach we can Bohrify the C*-algebra $\frk{A}$. This means that we make it commutative
by considering it as a copresheaf over the poset of commutative subalgebras. The difference with the copresheaf approach is that the
Bohrification takes place internal in the topos  $[\mathcal{K}(\mathcal{M}),\Set]$, instead of the topos $\Set$. We obtain a commutative
C*-algebra internal to a topos (which in turn is internal to a functor category), and, using the ideas of the preceding paragraphs, we
describe the points of the Gelfand spectrum of this commutative C*-algebra.

Instead of the Haag--Kastler formalism, we could have considered the more general and more recent locally covariant quantum field
theories~\cite{BrunettiFredenhagenVerch}. This amounts to replacing the poset $\mathcal{K}(\mathcal{M})$ by a more complicated category of
manifolds and
embeddings (which is no longer a poset). Although Bohrification of the locally covariant field theories can be described using the same
ideas of Subsection~\ref{sec:functor}, we stick with the Haag--Kastler formalism, as this makes the presentation a bit easier.

An internal unital commutative C*-subalgebra of $\frk{A}$ is simply a subobject $\frk{C}$ of
$\frk{A}$ such that for each $O\in\mathcal{K}(\mathcal{M})$, $\frk{C}(O)$ is a commutative unital
C*-algebra in $\mathbf{Set}$. These internally defined commutative C*-subalgebras form a poset $\mathcal{C}(\frk{A})$ in
$[\mathcal{K}(\mathcal{M}),\mathbf{Set}]$ and we can consider the
(internal) functor category over this poset. By using iterated forcing (see Section~\ref{sub:toppreshsh}),
we can describe this functor category within a functor category using a single Grothendieck topos over $\mathbf{Set}$, given by the
site $\mathcal{K}(\mathcal{M})\ltimes\mathcal{C}(\frk{A})$. In this (composite) topos, the Bohrified net is given
by the functor $(O,\frk{C})\mapsto \frk{C}(O)$, where $O\in\mathcal{K}(\mathcal{M})$ and $\frk{C}$ is
a commutative unital C*-subalgebra of $\frk{A}|_{\uparrow O}$. Before considering Gelfand Duality, we first simplify the topos in which we
are working. Instead of labeling the objects of the base category by subalgebras $\frk{C}$ of $\frk{A}|_{\uparrow O}$, we only concentrate
on the part $\frk{C}(O)$. This motivates using the topos $[\mathcal{P},\mathbf{Set}]$, where the poset $\mathcal{P}$ is defined as follows:
an element $(O,C)\in\mathcal{P}$, consists of an $O\in\mathcal{K}(\mathcal{M})$ and a $C\in\mathcal{C}_{O}:=\mathcal{C}(\frk{A}(O))$, and
the order relation is given by
\begin{equation*}
(O_{1},C_{1})\leq(O_{2}, C_{2})\ \ \text{iff}\ \ O_{1}\subseteq O_{2},\ \ C_{1}\subseteq C_{2}.
\end{equation*}
We are interested in the unital commutative C*-algebra $\ua:\mathcal{P}\to\mathbf{Set}$, $(O,C)\mapsto C$ in the topos $[\mathcal{P},\mathbf{Set}]$. Note that $\ua((O_{1},C_{1})\leq(O_{2}, C_{2}))$ is the inclusion map $C_{1}\hookrightarrow C_{2}$.

Next, we want to compute the points of the locale $\Sigma$, which is the external description of the Gelfand spectrum of $\ua$. An easy way
to do this is by using the reasoning in~\cite{Spitters:SpaceMOSNCA}. As a category, the topos $[\mathcal{P},\mathbf{Set}]$ is equivalent to
the topos $\Sh(\mathcal{P}_{\uparrow})$, where $\mathcal{P}_{\uparrow}$ is the set $\mathcal{P}$, equipped with the Alexandrov upset
topology. As in~\cite{Moer:CtsFibILT} we can find a site $\mathcal{P}\ltimes\uS$ such that $\Sh_{\Sh(\mathcal{P})}(\uS)$, the topos of
sheaves over $\uS$, internal to $\Sh(\mathcal{P}_{\uparrow})$ is equivalent to $\Sh(\mathcal{P}\ltimes\uS)$. The locale $\Sigma$ is the
locale generated by the posite $\mathcal{P}\ltimes\uS$. We use the posite description $\mathcal{P}\ltimes\uS$ in order to find the points.
We start with the functor $\underline{L}$, the distributive lattice object in $[\mathcal{P},\mathbf{Set}]$, given by
\begin{equation*}
\underline{L}:\mathcal{P}\to\mathbf{Set},\ \ \underline{L}(O,C)=L_{C},
\end{equation*}
\begin{equation*}
\underline{L}((O_{1},C_{1})\leq(O_{2}, C_{2})): L_{C_{1}}\to L_{C_{2}},\ [a]_{C_{1}}\mapsto[a]_{C_{2}}.
\end{equation*}
The elements of $\aqftcov$ are triples $(O,C,[a]_{C})$, where $O\in\mathcal{K}(\mathcal{M})$, $C\in\mathcal{C}_{O}$ and $[a]_{C}\in L_{C}$. The order of this poset is given by
\begin{equation*}
(O_{1},C_{1},[a_{1}]_{C_{1}})\leq(O_{2},C_{2},[a_{2}]_{C_{2}}),\ \ \text{iff}\ \ O_{2}\subseteq O_{1},\ C_{2}\subseteq C_{1},\ \ [a_{1}]_{C_{1}}\leq[a_{2}]_{C_{1}}.
\end{equation*}
The poset $\aqftcov$ is equipped with the following covering relation $\cov$, which is inherited from the covering relation $\underline{\cov}$, exploiting the fact that we are working over $\mathcal{P}$. We have a covering $(O,C,[a]_{C})\cov W$ iff for $W_{0}=\{[b]_{C}\in L_{C}\mid (O,C,[b]_{C})\in W\}$, the condition $[a]_{C}\cov W_{0}$ holds in $L_{C}$. Note that the covering relation on $\aqftcov$ is completely described in terms of covering relations on the $L_{C}$.

A point $\sigma$ of the external spectrum $\Sigma$ corresponds to a completely prime filter of $\aqftcov$. Recall that a filter $\sigma$ is a nonempty, upward closed and lower directed subset of $\aqftcov$, and that $\sigma$ is completely prime if it satisfies
\begin{equation*}
(O,C,[a]_{C})\in\sigma\ \ \text{and}\ \ (O,C,[a]_{C})\cov W,\ \ \text{implies}\ \ U\cap\sigma\neq\emptyset.
\end{equation*}

Let $\sigma$ be a point of $\Sigma$. It is straightforward to show that
\begin{equation*}
\mathcal{R}=\{O\in\mathcal{K}(\mathcal{M})\mid\exists C\in\mathcal{C}_{O}, \exists [a]_{C}\in L_{C},\ \text{s.t.}\ (O,C,[a]_{C})\in\sigma\}
\end{equation*}
is an ideal of $\mathcal{K}(\mathcal{M})$. Fix any $O\in\mathcal{R}$ and consider the set
\begin{equation*}
\mathcal{I}_{O}=\{C\in\mathcal{C}_{O}\mid\exists [a]_{C}\in L_{C}\ \text{s.t.}\ (O,C,[a]_{C})\in\sigma\}.
\end{equation*}
For any $O\in\mathcal{R}$, $\mathcal{I}_{O}$ is an ideal of $\mathcal{C}_{O}$. For a pair $O\in\mathcal{R}$ and $C\in\mathcal{I}_{O}$, define
\begin{equation*}
\sigma_{O,C}:=\{[a]_{C}\in L_{C}\mid (O,C,[a]_{C})\in\sigma\}.
\end{equation*}
As in~\cite{Spitters:SpaceMOSNCA}, it can be shown that $\sigma_{O,C}$ is a completely prime filter of $L_{C}$. A completely prime filter
$\sigma_{O,C}$ on $L_{C}$ corresponds to a unique point $\lambda(O,C)$ of the Gelfand spectrum $\Sigma_{C}$.

Next, we show how for different $O\in\mathcal{R}$, $C\in\mathcal{I}_{O}$, the $\lambda(O,C)\in\Sigma_{C}$ are related. Let, for some fixed $O\in\mathcal{R}$, $D\subset C$ in $\mathcal{C}_{O}$. Let $a\in D^{+}$ and assume that $[a]_{C}\in\sigma_{O,C}$. By the order on $\aqftcov$,
\begin{equation*}
(O,C,[a]_{C})\leq(O,D,[a]_{D})\in\sigma,
\end{equation*}
where we used that $\sigma$ is a filter, and therefore it is upward closed. For any $a\in D^{+}$, if $[a]_{C}\in\sigma_{O,C}$, then $[a]_{D}\in\sigma_{O,D}$. The filter $\sigma_{O,D}$ can be viewed as a frame map $\sigma_{O,D}:\opens\Sigma_{C}\to\underline{2}$ mapping the open $X^{D}_{a}=\{\lambda\in\Sigma_{D}\mid\langle\lambda,a\rangle>0\}$ to $1$ iff $\lambda(O,D)\in X^{D}_{a}$, iff $[a]_{D}\in\sigma_{O,D}$. If $\rho_{CD}:\Sigma_{C}\to\Sigma_{D}$ is the restriction map, then $\sigma_{O,C}\circ\rho^{-1}_{CD}:\opens\Sigma_{D}\to 2$ corresponds to the point $\lambda(O,C)|_{D}$. At the level of points of $\Sigma_{D}$, the implication
\begin{equation*}
\forall a\in D^{+},\ \ [a]_{C}\in\sigma_{O,C}\Rightarrow [a]_{D}\in\sigma_{O,D}
\end{equation*}
translates to:
\begin{equation*}
\forall a\in D^{+},\ \ \left(\sigma_{O,C}(X^{C}_{a})=1\right)\Rightarrow\left(\sigma_{O,D}(X^{D}_{a})=1\right).
\end{equation*}
As the $X^{D}_{a}$ form a basis of the Hausdorff space $\Sigma_{D}$, and $\rho^{-1}(X^{D}_{a})=X^{C}_{a}$, this can only mean that $\sigma_{O,D}=\sigma_{O,C}\circ\rho^{-1}_{DC}$. In other words, whenever $D\subseteq C$, one has $\lambda(O,D)=\lambda(O,C)|_{D}$.

Assume that $O'\subset O$ in $\mathcal{K}(\mathcal{M})$ and that $C\in\mathcal{C}_{O'}$. In $\aqftcov$,
\begin{equation*}
\forall [a]_{C}\in L_{C},\ \ (O,C,[a]_{C})\leq(O',C,[a]_{C}).
\end{equation*}
If $[a]_{C}\in\sigma_{O,C}$, then by the filter property of $\sigma$, $[a]_{C}\in\sigma_{O',C}$. We conclude that if $O'\subseteq O$ in $\mathcal{R}$ and $C\in\mathcal{C}_{O'}$, then $\lambda(O',C)=\lambda(O,C)$. Hence:

\begin{theorem}
A point $\sigma$ of $\Sigma$ is described by a triple $(\mathcal{R},\mathcal{I}_{\mathcal{R}},\lambda_{\mathcal{R},\mathcal{I}})$, where:
\begin{itemize}
\item $\mathcal{R}$ is an ideal in $\mathcal{K}(\mathcal{M})$.
\item The function $\mathcal{I}_{\mathcal{R}}$ associates to each $O\in\mathcal{R}$, an ideal $\mathcal{I}_{O}$ of $\mathcal{C}_{O}$ satisfying two conditions. Firstly, if $O_{1}\subseteq O_{2}$, then $\mathcal{I}_{O_{2}}\cap\mathcal{C}_{O_{1}}\subseteq\mathcal{I}_{O_{1}}$. Secondly, if $C_{i}\in\mathcal{I}_{O_{i}}$, where $i\in\{1,2\}$, then there is an $O\in\mathcal{R}$ and a $C\in\mathcal{I}_{O}$ such that $O_{i}\subseteq O$ and $C_{i}\subseteq C$.
\item The function $\lambda_{\mathcal{R},\mathcal{I}}$ associates to each $O\in\mathcal{R}$ and $C\in\mathcal{I}_{O}$, an element $\lambda_{O,C}\in\Sigma_{C}$, such that if $O_{1}\subseteq O_{2}$ and $C_{1}\subseteq C_{2}$, then $\lambda_{O_{1},C_{1}}=\lambda_{O_{2},C_{2}}|_{C_{1}}$.
\end{itemize}
\end{theorem}

The two conditions in the second bullet point are included to ensure that the set
\begin{equation*}
\mathcal{I}=\{(O,C)\in\mathcal{P}\mid O\in\mathcal{R},C\in\mathcal{I}_{O}\}
\end{equation*}
is an ideal of $\mathcal{P}$. Mathematically, the theorem would look more elegant if it were formulated in terms of ideals of $\mathcal{P}$
instead of using pairs $(\mathcal{R},\mathcal{I}_{\mathcal{R}})$, but that description would miss an important physical point. Namely, a
spacetime point $x\in M$ corresponds to a specific filter of $\mathcal{K}(\mathcal{M})$, consisting of all $O\in\mathcal{K}(\mathcal{M})$
containing $x$. However, a point $\sigma$ of $\Sigma$ is labelled by an ideal $\mathcal{R}$ of $\mathcal{K}(\mathcal{M})$ and not by a
filter. If we want the points of the phase space to be indexed by points of the spacetime $\mathcal{M}$, it might be interesting to look at
the contravariant functor $\underline{\Sigma}:\mathcal{P}^{op}\to\mathbf{Set}$, $(O,C)\mapsto\Sigma_{C}$. Remaining on the topic of applying
topos approaches to quantum theory to algebraic quantum field theory, this functor $\underline{\Sigma}$ is also interesting when we consider
the work by Nuiten on this subject \cite{nuiten}. Nuiten investigates relations between independence conditions on nets of operator algebras
 on the one hand, and sheaf conditions on the corresponding Bohrified functors on the other. As argued in \cite{Wolters/Halvorson}, the
sheaf condition of Nuiten can be viewed as a sheaf condition on the functor $\underline{\Sigma}:\mathcal{P}^{op}\to\mathbf{Set}$ (although
strictly speaking the covering relation which is involved does not satisfy all conditions for the basis of a Grothendieck topology).

\section{Acknowledgements}

The authors would like to thank Klaas Landsman for his comments, which greatly improved this paper,
and also an anonymous referee for his or her comments.
\bibliographystyle{eptcs}
\bibliography{geoBS}
\end{document}